\documentclass[12pt,a4paper]{article}
\usepackage[utf8]{inputenc}
\usepackage{amsmath}
\usepackage{amsfonts}
\usepackage{amssymb}
\usepackage{natbib}
\usepackage{url}
\usepackage[left=1.2cm, right=1.2cm]{geometry}
\usepackage{color}
\usepackage{hyperref}
\usepackage{authblk}

\newtheorem{condition}{Condition}
\newtheorem{theorem}{Theorem}
\newtheorem{proposition}{Proposition}
\newtheorem{lemma}{Lemma}
\newtheorem{example}{Example}
\newtheorem{definition}{Definition}

\newenvironment{proof}{\paragraph{Proof:}}{\hfill$\square$}

\author[1]{Andrea Rotnitzky \thanks{arotnitzky@utdt.edu}}
\affil[1]{Department of Economics, Universidad Torcuato Di Tella and CONICET}
\author[2]{Ezequiel Smucler \thanks{esmucler@utdt.edu}}
\affil[2]{Department of Mathematics and Statistics, Universidad Torcuato Di Tella}
\author[3]{James M. Robins \thanks{robins@hsph.harvard.edu}}
\affil[3]{Department of Epidemiology and Biostatistics, Harvard T.H. Chan School of Public Health}

\title{Characterization of parameters with a mixed bias property}

\date{}

\begin{document}

\maketitle
\begin{abstract}
In this article we study a class of parameters with the so-called `mixed
bias property'. For parameters with this property, the bias of the
semiparametric efficient one step estimator is equal to the mean of the
product of the estimation errors of two nuisance functions. In
non-parametric models, parameters with the mixed bias property admit
so-called rate doubly robust estimators, i.e. estimators that are consistent
and asymptotically normal when one succeeds in estimating both nuisance
functions at sufficiently fast rates, with the possibility of trading off
slower rates of convergence for the estimator of one of the nuisance
functions with faster rates for the estimator of the other nuisance. We show
that the class of parameters with the mixed bias property strictly includes
two recently studied classes of parameters which, in turn, include many
parameters of interest in causal inference. We characterize the form of
parameters with the mixed bias property and of their influence functions.
Furthermore, we derive two functional moment equations, each being solved at
one of the two nuisance functions, as well as, two functional loss
functions, each being minimized at one of the two nuisance functions. These
loss functions can be used to derive loss based penalized estimators of the
nuisance functions.
\end{abstract}

\section{Introduction}

\label{sec:setup}

Suppose that we are given a sample $\mathcal{D}_{n}$ of $n$ i.i.d. copies of
a random vector $O$ with law $P$ which is known to belong to  $\mathcal{M}=\left\{ P_{\eta }:\eta \in \mathbf{\eta }\right\} $ where $%
\mathbf{\eta }$ is a large, non-Euclidean, parameter space. Our goal is to
estimate the value taken by a scalar parameter $\chi \left( \eta \right) $
at $P$. Suppose $O$ includes a vector $Z$ with sample space $\mathcal{Z}$ $%
\subset R^{d}.$ We are interested in parameters $\chi \left( \eta \right) $
which cannot be estimated without estimating some unknown function of the
covariates $Z,$ such as a conditional mean given $Z$ or a density of $Z.$

Given an initial estimator $\widehat{\eta }$, the plug-in estimator $\chi \left( \widehat{%
\eta }\right) $ is a natural choice for estimating $\chi \left( \eta \right)
.$ However, except for special estimators $\widehat{\eta }$ targeted to
specific parameters $\chi \left( \eta \right) $, $\chi \left( \widehat{\eta }%
\right) $ is not $\sqrt{n}$- consistent. A strategy for reducing the bias of 
$\chi \left( \widehat{\eta }\right) $ is to subtract from it an estimate $-%
\mathbb{P}_{n}\chi _{\widehat{\eta }}^{1}$ of its first order bias where for each $\eta ,$ $\chi _{\eta }^{1}\equiv \chi
_{\eta }^{1}\left( O\right) $ is an adequately chosen random variable and $%
\mathbb{P}_{n}h$ is the empirical mean operator $n^{-1}\sum_{i=1}^{n}h\left(
O_{i}\right) $. This strategy yields the one step estimator $\widehat{\chi }%
\equiv \chi \left( \widehat{\eta }\right) +\mathbb{P}_{n}\chi _{\widehat{%
\eta }}^{1}.$ A good choice for $\chi _{\eta }^{1}$ is a so called influence
function of $\chi \left( \eta \right)$. See for example \cite{neweystep1}, \cite{neweystep2} and \cite{Robins17}. Heuristically, this choice is
guided by the following analysis. Write 
\begin{equation*}
\sqrt{n}\left\{ \widehat{\chi }-\chi \left( \eta \right) \right\} =\sqrt{n}%
\left\{ \chi \left( \widehat{\eta }\right) -\chi \left( \eta \right)
+E_{\eta }\left( \chi _{\widehat{\eta }}^{1}\right) \right\} +\mathbb{G}%
_{n}\left( \chi _{\widehat{\eta }}^{1}-\chi _{\eta }^{1}\right) +\mathbb{G}%
_{n}\left( \chi _{\eta }^{1}\right) .
\end{equation*}%
where $E_{\eta }\left( \chi _{\widehat{\eta }}^{1}\right) \equiv \int \chi _{%
\widehat{\eta }}^{1}\left( o\right) dP_{\eta }\left( o\right) $ and $\mathbb{%
G}_{n}\left( \chi _{\widehat{\eta }}^{1}\right) \equiv $ $\sqrt{n}\mathbb{P}%
_{n}\left\{ \chi _{\widehat{\eta }}^{1}-E_{\eta }\left( \chi _{\widehat{\eta 
}}^{1}\right) \right\} .$

The term $\mathbb{G}_{n}\left( \chi _{\eta }^{1}\right) $ is $\sqrt{n}$
times the sample average of mean zero random variables, so it converges to a
normal distribution. On the other hand, if model $\mathcal{M}$ is not too
big, then for estimators $\widehat{\eta }$ converging to $\eta $ one may
expect $\mathbb{G}_{n}\left( \chi _{\widehat{\eta }}^{1}-\chi _{\eta
}^{1}\right) $ to be $o_{p}\left( 1\right) $. One can make this term $%
o_{p}\left( 1\right) $ even without restrictions on model size by employing
the following strategy known as cross-fitting \citep{schick1986, vaart-book98, ayyagariphd, zheng2011cross} . First split sample $%
\mathcal{D}_{n}$ into two samples, next compute $\widehat{\eta }$ from one
subsample and the one step estimator from the other subsample. Next, compute
a second one step estimator by repeating the procedure but switching the
roles of the two subsamples. Finally, compute the estimator $\widetilde{\chi 
}$ of $\chi \left( \eta \right) $ as the average of both one step
estimators. Convergence of $\sqrt{n}\left\{ \widetilde{\chi }-\chi \left(
\eta \right) \right\} $ to a mean zero normal distribution thus depends
essentially solely on 
\begin{equation}
\chi \left( \widehat{\eta }\right) -\chi \left( \eta \right) -E_{\eta
}\left( \chi _{\widehat{\eta }}^{1}\right)  \label{eq:expansion}
\end{equation}%
being $o_{p}\left( n^{-1/2}\right) .$ This last requirement suggests that we choose $\chi _{\eta }^{1}$ to
be an influence function of $\chi \left( \eta \right) $. This is because for
such choice $E_{\eta }\left( \chi _{\widehat{\eta }}^{1}\right) $
acts like minus the derivative of $\chi \left( \eta
\right) $ in the direction $\widehat{\eta }-\eta $. Consequently $\left( %
\ref{eq:expansion}\right) $ acts like the residual from a first order
Taylor's expansion of $\chi \left( \eta \right) ,$ and hence is of order $%
O\left( \left\Vert \widehat{\eta }-\eta \right\Vert ^{2}\right) .$ Thus, for
estimators $\widehat{\eta }$ such that $n^{1/4}\left\Vert \widehat{\eta }%
-\eta \right\Vert =o_{p}\left( 1\right) $ for some norm$\left\Vert \cdot
\right\Vert ,$ $\left( \ref{eq:expansion}\right) $ should be of order $%
o_{p}\left( n^{-1/2}\right) .$ See Chapter 25 in \cite{vaart-book98} for the
definition of influence functions. Parameters that admit influence functions
are called regular parameters. Such parameters have a unique influence
function if the model $\mathcal{M}$ is non-parametric. By non-parametric we
mean that the closed linear span of the scores for all parametric submodels
at $P$ of model $\mathcal{M}$ is equal to $L_{2}\left( P\right) .$
Throughout we will assume that $\mathcal{M}$ is non-parametric, that $\chi
\left( \eta \right) $ is regular and that $\chi _{\eta }^{1}$ is
the unique influence function of $\chi \left( \eta \right) $.

Many parameters $\chi \left( \eta \right) $ of interest in Causal Inference
and Econometrics have influence functions which satisfy the following
property.

\begin{definition}[Mixed bias property]
\label{def:mixed_bias} For each $\eta $\ there exist functions $a\left(
Z\right) \equiv a\left( Z;\eta \right) $ and $b\left( Z\right) \equiv
b\left( Z;\eta \right) $ such that for any $\eta ^{\prime }:$ 
\begin{equation}
\chi \left( \eta ^{\prime }\right) -\chi \left( \eta \right) +E_{\eta
}\left( \chi _{\eta ^{\prime }}^{1}\right) =E_{\eta }\left[ S_{ab}\left\{
a^{\prime }\left( Z\right) -a\left( Z\right) \right\} \left\{ b^{\prime
}\left( Z\right) -b\left( Z\right) \right\} \right]  \label{bias1}
\end{equation}%
where $a^{\prime }\left( Z\right) \equiv a\left( Z;\eta ^{\prime }\right)
,b^{\prime }\left( Z\right) \equiv b\left( Z;\eta ^{\prime }\right) $ and $%
S_{ab}\equiv s_{ab}\left( O\right) $ and $o\rightarrow s_{ab}\left( o\right) 
$ is a known function, i.e. it that does not depend on $\eta .\medskip $
\end{definition}

As we will see in the next section, the mixed bias property implies that $%
\chi \left( \eta \right) +\chi _{\eta }^{1}$ depends on $\eta $ only through 
$a$ and $b$ and, consequently, the one step estimator depends on $\widehat{%
\eta }$ only through estimators $\widehat{a}$ and $\widehat{b}$. The
property implies that for estimators $\widehat{a}$ and $\widehat{b}$
satisfying $\int \left\{ \widehat{a}\left( z\right) -a\left( z\right)
\right\} ^{2}dP_{\eta }\left( z\right) =O_{p}\left( \gamma _{a,n}\right) $
and $\int \left\{ \widehat{b}\left( z\right) -b\left( z\right) \right\}
^{2}dP_{\eta }\left( z\right) =O_{p}\left( \gamma _{b,n}\right) $, equation $%
\left( \ref{eq:expansion}\right) $ is of order $O_{P}\left( \gamma
_{a,n}\gamma _{b,n}\right) .$ This in turn implies that, when cross-fitting
is employed, $\widetilde{\chi }$ has the so-called rate double robustness
property in that $\sqrt{n}\left\{ \widetilde{\chi }-\chi \left( \eta \right)
\right\} $ converges to a mean zero Normal distribution if $\gamma
_{a,n}=o\left( 1\right) ,$ $\gamma _{b,n}=o\left( 1\right) $ and $\gamma
_{a,n}\gamma _{b,n}=o\left( n^{-1/2}\right) .$ Because the rates of
convergence $\gamma _{a,n}$ and $\gamma _{b,n}$ of estimators $\widehat{a}$
and $\widehat{b}$ depend on the complexities of $a$ and $b,$ then for
parameters $\chi \left( \eta \right) $ satisfying the mixed bias property, $%
\sqrt{n}\left\{ \widetilde{\chi }-\chi \left( \eta \right) \right\} $ is
asymptotically normal even if one of the functions $a$ or $b$ is very
complex so long as the other is simple enough.

Recent articles have
identified two distinct classes of parameters $\chi \left( \eta \right) $
with the mixed bias property and gave many examples of parameter of interest in causal inference and econometrics, including the examples in \S\ \ref{ch:examples} below. 
The first class, described in  \cite{Robins08Higher} is comprised of parameters  with influence function of the form 
$
\chi _{\eta }^{1}=S_{ab}a\left( Z\right) b\left( Z\right) +S_{a}a\left(
Z\right) +S_{b}b\left( Z\right) +S_{0}-\chi \left( \eta \right)
$
where $S_{a}$ and $S_{b}$ are statistics. The second class, described in
\cite{NeweyCherno19} (see also \cite%
{hirshberg} and \cite{cherno2}) is comprised of parameters of the form $\chi \left( \eta
\right) =E_{\eta }\left\{ d\left( O,a\right) \right\} $ where $a\left(
Z\right) \equiv E_{\eta }\left( Y|Z\right) $ and $d\left( O,a\right) $ is
such that the map $h\in L_{2}\left( P_{\eta ,Z}\right) \rightarrow E_{\eta
}\left\{ d\left( O,h\right) \right\} $ is continuous and affine linear.


In this paper we will characterize the form of the influence function, under
a non-parametric model $\mathcal{M}$, of any parameter satisfying the mixed
bias property. We will show that the class of parameters satisfying the
mixed bias property strictly includes the union of the \cite{Robins08Higher}
and \cite{NeweyCherno19} classes. Furthermore, we will show that neither the
class \cite{Robins08Higher} nor that of \cite{NeweyCherno19} is contained in
the other. We will also show that, under mild regularity conditions,
parameters that satisfy the mixed bias property are necessarily of the form 
\begin{equation}
\chi \left( \eta \right) =E_{\eta }\left\{ m_{1}\left( O,a\right) \right\}
+E_{\eta }\left( S_{0}\right) =E_{\eta }\left\{ m_{2}\left( O,b\right)
\right\} +E_{\eta }\left( S_{0}\right)   \label{eq:parameter}
\end{equation}%
for some statistic $S_{0}$, and some $m_{1}$ and $m_{2}$ such that the maps $%
h\in \mathcal{A}\rightarrow m_{1}\left( O,h\right) $ and $h\in \mathcal{B}%
\rightarrow m_{2}\left( O,h\right) $ are linear, where $\mathcal{A}\equiv
\left\{ a\left( Z;\eta \right) :\eta \in \mathbf{\eta }\right\} $ and $%
\mathcal{B}\equiv \left\{ b\left( Z;\eta \right) :\eta \in \mathbf{\eta }%
\right\} .$ In addition, we will prove a number of results about the
structure of $a$ and/or $b$ in special cases. In particular, we will show
that, under mild regularity conditions, when $a$ does not depend on the
marginal distribution of $Z,$ then, up to regularity conditions, a necessary
and sufficient condition for $\chi \left( \eta \right) $ to have the mixed
bias property is that $\chi \left( \eta \right) =E_{\eta }\left\{
m_{1}\left( O,a\right) \right\} +E_{\eta }\left( S_{0}\right) $ for a
statistic $S_{0},$ a linear map $h\in \mathcal{A}\rightarrow m_{1}\left(
O,h\right) $ and $a\left( Z\right) $ a ratio of two conditional means given $%
Z.$ We will also show that for parameters $\chi \left( \eta \right) $ that
satisfy the mixed bias property the influence function naturally yields two
loss functions whose expectations are minimized at $a$ and $b$ respectively.
These loss functions can then be used to construct loss-based
machine-learning estimators of $a$ and $b$ such as support vector machine
estimators \citep{svm}.

Our work is related to \cite{robinscomment} and \cite{local-dr}. These papers discuss sufficient
conditions for the existence of, so called, doubly robust estimating
functions. A key distinction of our work is that, unlike these papers, we do
not assume that the parameter solves a population moment equation, rather we
deduce this fact from the primitive condition of the mixed bias property.

\section{Characterization of the influence functions with the mixed bias
property}

Our first result establishes that for parameters $\chi \left( \eta \right) $ that satisfy the mixed bias
property, $\chi \left( \eta \right) +\chi _{\eta }^{1}$ depends on $\eta $
only through $a$ and $b.$

\begin{proposition}
\label{prop:0} If $\chi \left(
\eta \right) $ satisfies the mixed bias property and the regularity
Condition \ref{cond:R1}, then $\chi
\left( \eta \right) +\chi _{\eta }^{1}$ $\ $depends on $\eta $ only through $%
a$ and $b.$
\end{proposition}

The Supplementary Material contains proofs of all the claims made in this article, the
regularity Conditions \ref{cond:R1} and \ref{cond:R2} invoked by them, and further examples of parameters with the mixed bias property. The
next Theorem characterizes the influence functions of parameters with the
mixed bias property.

\begin{theorem}
\label{theo:1}If $\chi \left( \eta
\right) $ satisfies the mixed bias property and the regularity Condition \ref%
{cond:R1} holds, then there exist a statistic $S_{0}$ and maps $h\in 
\mathcal{A}\rightarrow $ $m_{1}\left( O,h\right) $ and $h\in \mathcal{B}%
\rightarrow m_{2}\left( O,h\right) $, independent of $\eta$, such that the maps $h\in \mathcal{A}%
\rightarrow E_{\eta }\left\{ m_{1}\left( O,h\right) \right\} $ and $h\in 
\mathcal{B}\rightarrow E_{\eta }\left\{ m_{2}\left( O,h\right) \right\} $
are linear and such that $\left( \ref{eq:parameter}\right) $ and 
\begin{equation}
\chi _{\eta }^{1}=S_{ab}a\left( Z\right) b\left( Z\right) +m_{1}(O,a) +m_{2}(O,b) +S_{0}-\chi \left( \eta \right) 
\label{eq:IF}
\end{equation}%
hold. Furthermore, for all $h\in \mathcal{A},$ $E_{\eta }\left\{
S_{ab}hb+m_{1}\left( O,h\right) \right\} =0$ and for all $h\in \mathcal{B},$ 
$E_{\eta }\left\{ S_{ab}ha+m_{2}\left( O,h\right) \right\} =0.$ In addition,
for any $h_{1},h_{2}\in \mathcal{A}$ and constants $\alpha _{1},\alpha _{2}$
such that $h\equiv \alpha _{1}h_{1}+\alpha _{2}h_{2}\in \mathcal{A}$ and
such that $m_{1}\left( O,h\right) ,m_{1}\left( O,h_{1}\right) \,$\ and $%
m_{1}\left( O,h_{2}\right) $ are in $L_{2}\left( P_{\eta }\right) $, it
holds that $m_{1}\left( O,h\right) =\alpha _{1}m_{1}\left( O,h_{1}\right)
+\alpha _{2}m_{1}\left( O,h_{2}\right) $ a.s.$\left( P_{\eta }\right) .$ In
particular, if for all $h\in \mathcal{A}$, $m_{1}\left( O,h\right) \in
L_{2}\left( P_{\eta }\right) $ then the map $h\in \mathcal{A}\rightarrow $ $%
m_{1}\left( O,h\right) $ is linear a.s.$\left( P_{\eta }\right) $. Likewise,
if for all $h\in \mathcal{B}$ it holds that $m_{2}\left( O,h\right) \in
L_{2}\left( P_{\eta }\right) $ then the map $h\in \mathcal{B}\rightarrow $ $%
m_{2}\left( O,h\right) $ is linear a.s.$\left( P_{\eta }\right) $.
\end{theorem}

Part (i) of the next result establishes that under a slightly stronger
requirement on $m_{1}$ and $m_{2}$ and some regularity conditions, the
reverse of Theorem \ref{theo:1} also holds. The theorem also establishes
several additional results that we will comment after its statement.

\begin{theorem}
\label{theo:2}Suppose that for each 
$\eta $\ there exist functions $a\left( Z\right) \equiv a\left( Z;\eta
\right) $ and $b\left( Z\right) \equiv b\left( Z;\eta \right) $ such that
the regularity Condition \ref{cond:R1} holds and such that the influence function of $\chi \left( \eta \right) $ is
of the form $\left( \ref{eq:IF}\right) $ for $m_{1}$ and $m_{2}$ that
satisfy that for each $\eta ,$ the maps 
\begin{equation*}
h\in L_{2}\left( P_{\eta ,Z}\right) \rightarrow E_{\eta }\left\{ m_{1}\left(
O,h\right) \right\} \text{ and }h\in L_{2}\left( P_{\eta ,Z}\right)
\rightarrow E_{\eta }\left\{ m_{2}\left( O,h\right) \right\} 
\end{equation*}%
are continuous and linear with Riesz representers $\mathcal{R}_{1}\left(
Z\right) $ and $\mathcal{R}_{2}\left( Z\right) $ respectively. Moreover,
suppose $E_{\eta }\left\{ m_{1}\left( O,a\right) \right\} $ and $E_{\eta
}\left\{ m_{2}\left( O,b\right) \right\} $ exist. Furthermore, suppose that
for each $\eta ,$ $E_{\eta }\left( S_{ab}|Z\right) a\left( Z\right) $ and $%
E_{\eta }\left( S_{ab}|Z\right) b\left( Z\right) $ are in $L_{2}\left(
P_{\eta ,Z}\right) $. Then,\newline
(i) the identity $\left( \ref{bias1}\right) $ holds for each $\eta ^{\prime }
$ such that $a^{\prime }\left( Z\right) \equiv a\left( Z;\eta ^{\prime
}\right) ,b^{\prime }\left( Z\right) \equiv b\left( Z;\eta ^{\prime }\right) 
$ satisfy that $a^{\prime }-a\in L_{2}\left( P_{\eta ,Z}\right) $ and $%
b^{\prime }-b\in L_{2}\left( P_{\eta ,Z}\right) $. \newline
(ii) for all $h\in L_{2}\left( P_{\eta ,Z}\right) $ it holds that $E_{\eta
}\left\{ S_{ab}ha+m_{2}\left( O,h\right) \right\} =0$ and $E_{\eta }\left\{
S_{ab}hb+m_{1}\left( O,h\right) \right\} =0$. \newline
(iii) if $E_{\eta }\left( S_{ab}|Z\right) \not=0$ a.s.$\left( P_{\eta
,Z}\right) ,$ then $a\left( Z\right) =-\mathcal{R}_{2}\left( Z\right)
/E_{\eta }\left( S_{ab}|Z\right) .$ Likewise, if $E_{\eta }\left(
S_{ab}|Z\right) \not=0$ a.s.$\left( P_{\eta ,Z}\right) ,$ then $b\left(
Z\right) =-\mathcal{R}_{1}\left( Z\right) /E_{\eta }\left( S_{ab}|Z\right)
.$\newline
(iv)\newline
(iv.a) if $a\in $ $L_{2}\left( P_{\eta ,Z}\right) $ or if $\ \exists
\varepsilon >0$ such that $\left( 1+t\right) a\in \mathcal{A}$ for $%
0<t<\varepsilon $ or for $-\varepsilon <t<0$ then $\chi \left( \eta \right)
=E_{\eta }\left\{ m_{2}\left( O,b\right) \right\} +E_{\eta }\left(
S_{0}\right) .$ \newline
(iv.b) Likewise, if $b\in $ $L_{2}\left( P_{\eta ,Z}\right) $ or if $\
\exists \varepsilon >0$ such that $\left( 1+t\right) b\in \mathcal{B}$ for $%
0<t<\varepsilon $ or for $-\varepsilon <t<0$ then $\chi \left( \eta \right)
=E_{\eta }\left\{ m_{1}\left( O,a\right) \right\} +E_{\eta }\left(
S_{0}\right) .$\newline
(iv.c) if the conditions of parts (iv.a) and (iv.b) hold then $\chi \left(
\eta \right) =-E_{\eta }\left( S_{ab}ab\right) +E_{\eta }\left( S_{0}\right)
.$ \newline
(v) if $a\in $ $L_{2}\left( P_{\eta ,Z}\right) ,$ $b\in $ $L_{2}\left(
P_{\eta ,Z}\right) $ and $E_{\eta }\left( S_{ab}|Z\right) >0$ a.s.$\left(
P_{\eta ,Z}\right) ,$ then  
\begin{equation*}
a=\arg \min_{h\in L_{2}\left( P_{\eta ,Z}\right) }E_{\eta }\left\{ S_{ab}%
\frac{h^{2}}{2}+m_{2}\left( O,h\right) \right\} \text{ and }b=\arg
\min_{h\in L_{2}\left( P_{\eta ,Z}\right) }E_{\eta }\left\{ S_{ab}\frac{h^{2}%
}{2}+m_{1}\left( O,h\right) \right\} 
\end{equation*}
\end{theorem}

Note that part (ii) of Theorem \ref{theo:2} provides unbiased moment
equations for $a$ and $b$ respectively without requiring that $a$ or $b$ be
in $L_{2}\left( P_{\eta ,Z}\right) $. \cite{NeweyCherno19} and \cite{ratemodel} exploit these moment equations to construct $\ell _{1}$ regularized estimators of the
nuisance functions. Part (iii) of the theorem
provides the formulae for $a$ and $b$ in terms of the Riesz representers of
the maps. Part (iv) shows that under a strengthening on the requirements on $%
a$ and $b,$ the representation in $\left( \ref{eq:parameter}\right) $ holds.
Note that the requirement that $\left( 1+t\right) b\in \mathcal{B}$ for $%
0<t<\varepsilon $ is rather mild. For instance, for $b\left( Z\right)
=1/P\left( D=1|Z\right) ,\,$as in example \ref{ex:MAR} below, the
requirement is satisfied since the only restriction\thinspace\ the elements $%
b^{\prime }$ of $\mathcal{B}$ satisfy is that for each $z,$ $b^{\prime
}\left( z\right) $ must be greater than or equal 1. Part (v) of the Theorem
could in principle be used to derive other machine learning, loss-based
estimators of these parameters, such as support vector machines \citep{svm}.

\section{Characterization of the nuisance functions}

An interesting question is what can be said about the restrictions that the
nuisance functions $a$ and $b$ of parameters with the mixed bias property
must satisfy. In this section we explore this question in the special case
in which $a$ does not depend on the marginal law of $Z$. We will show that
such $a$ must be a ratio of conditional expectations given $Z.$

\begin{proposition}
\label{prop:2} Suppose that the
parameter $\chi \left( \eta \right) $ satisfies the mixed bias property, the
regularity Conditions \ref{cond:R1} and \ref{cond:R2} hold and $E_{\eta }\left( S_{ab}|Z\right) \not=0$ a.s. $\left(
P_{\eta ,Z}\right) .$ If $a$ depends on $\eta $ only through the law of $O|Z,
$ then there exists a statistic $q\left( O\right) $ such that $a\left(
Z\right) =-$ $E_{\eta }\left\{ q\left( O\right) |Z\right\} /E_{\eta }\left(
S_{ab}|Z\right) $. Furthermore, the influence function of $\chi \left( \eta
\right) $ satisfies $\left( \ref{eq:IF}\right) $ for some linear map $h\in 
\mathcal{A}\rightarrow m_{1}\left( O,h\right) $ and $m_{2}\left( O,b\right)
=q\left( O\right) b.$
\end{proposition}

\begin{proposition}
\label{prop:3} Suppose that $a\left( Z\right) =-$ $E_{\eta }\left\{ q\left(
O\right) |Z\right\} /E_{\eta }\left( S_{ab}|Z\right) $ is in $L_{2}\left(
P_{\eta ,Z}\right) .$ Suppose also that the map $h\in L_{2}\left( P_{\eta
,Z}\right) \rightarrow E_{\eta }\left\{ m_{1}\left( O,a\right) \right\} $ is
linear and continuous with Riesz representer $\mathcal{R}_{1}\left( Z\right)
.$ Then, $\chi \left( \eta \right) =E_{\eta }\left\{ m_{1}\left( O,a\right)
\right\} +E_{\eta }\left( S_{0}\right) $\ has influence function that
satisfies $\left( \ref{eq:IF}\right) $ with $m_{2}\left( O,b\right)
=q\left( O\right) b$ and $b\left( Z\right) =-{\mathcal{R}_{1}\left( Z\right) 
}/{E_{\eta }\left( S_{ab}|Z\right) }$. In addition, if $E_{\eta }\left\{
q\left( O\right) |Z\right\} \in L_{2}\left( P_{\eta ,Z}\right) $, then $\chi
\left( \eta \right) $ has the mixed bias property.
\end{proposition}

For a given parameter $\chi \left( \eta \right) $ there can exist more than
one function $a\left( Z\right) \equiv a\left( Z,\eta \right) $ independent
of the law of $Z$ such that the mixed bias property holds for some $b\left(
Z\right) \equiv b\left( Z,\eta \right) .$ An instance is the parameter in
Example \ref{ex:CC} below, since $a\left( Z\right) $ can be either $E_{\eta
}\left( Y|Z\right) $ or $E_{\eta }\left( D|Z\right) .$ However, in that
example, if $a=E_{\eta }\left( Y|Z\right) $ then $b=E_{\eta }\left(
D|Z\right) $ and vice versa, if $a=E_{\eta }\left( D|Z\right) $ then $%
b=E_{\eta }\left( Y|Z\right) .$ An open question is whether or not there
exist two distinct triplets $\left( S_{ab},a,b\right) $ and $\left(
S_{ab}^{\ast },a^{\ast },b^{\ast }\right) $ with $\left( a,b\right)
\not=\left( a^{\ast },b^{\ast }\right) $ such that the parameter $\chi
\left( \eta \right) \,\,$satisfies the mixed bias property for both
triplets. This is important because if such distinct triplets existed, then
there would exist two different pairs of nuisance functions of the same
covariate $Z$ that one could choose to estimate in order to construct rate
doubly robust estimators of $\chi \left( \eta \right) $.

In the preceding propositions we have assumed a given partition of the data $%
O$ into a given `covariate' vector $Z$ and the remaining variables in $O.$
Interestingly, there exist parameters $\chi \left( \eta \right) $ that
satisfy the mixed bias property for two different partitions of $O,$ one
with `covariate' vector $Z$ and another with a different `covariate' vector $%
Z^{\ast }.$ Specifically, in Example \ref{ex:MAR} we show that for $\chi
\left( \eta \right) $ equal to the mean of an outcome missing at random,
there exist two possible partitions of $O,$ into two different `covariate'
vectors $Z$ and $Z^{\ast },$ and functions $a^{\ast }\left( Z^{\ast },\eta
\right) $ and $b^{\ast }\left( Z^{\ast },\eta \right) $ different from $%
a\left( Z\right) $ and $b\left( Z\right) ,$ $a\left( Z,\eta \right) $ and $%
a^{\ast }\left( Z^{\ast },\eta \right) $ depending on $\eta $ only through
the law of $O|Z$ such that for all $\eta $ and $\eta ^{\prime }$ 
\begin{equation*}
S_{ab}\left\{ a\left( Z,\eta \right) -a\left( Z,\eta ^{\prime }\right)
\right\} \left\{ b\left( Z,\eta \right) -b\left( Z,\eta ^{\prime }\right)
\right\} =S_{ab}^{\ast }\left\{ a^{\ast }\left( Z^{\ast },\eta \right)
-a^{\ast }\left( Z^{\ast },\eta ^{\prime }\right) \right\} \left\{ b^{\ast
}\left( Z^{\ast },\eta \right) -b^{\ast }\left( Z^{\ast },\eta ^{\prime
}\right) \right\} .
\end{equation*}
Consequently, the parameter $\chi \left( \eta \right) $ satisfies the mixed
bias property for the functions $a$ and $b,$ but also for the functions $%
a^{\ast }$ and $b^{\ast }.$ In this example, $S_{ab}$ is not a constant but $%
S_{ab}^{\ast }$ is a constant, so in view of part (i) of Proposition \ref%
{prop:2}, $a\left( Z\right) $ is a ratio of two conditional expectations
given $Z,$ whereas $a^{\ast }$ is a conditional expectation of a specific
statistic $q\left( O\right) $ given $Z^{\ast }.$ This example raises the
following interesting question: suppose that $\chi \left( \eta \right) $
satisfies the mixed bias property for a function $a\left( Z\right) $ that is
a strict ratio of two conditional expectations given $Z,$ is it always
possible to find a different covariate vector $Z^{\ast }$ such that $\chi
\left( \eta \right) $ satisfies the mixed bias property for a function $%
a^{\ast }\left( Z^{\ast }\right) $ that is a conditional mean of a statistic
given $Z^{\ast }?$ The answer  is negative, as Example \ref%
{ex:MNAR} below illustrates. This example proves that the class of
parameters that satisfy the mixed bias property strictly includes the class
considered in \cite{NeweyCherno19}.

We conclude our analysis answering the question of whether a
characterization exist of nuisance functions that depend on the marginal law
of $Z$ and possibly also on the law of $O|Z$. The answer to this question is
negative. This can be understood from Proposition \ref{prop:3} because when
the map $h\in L_{2}\left( P_{\eta ,Z}\right) \rightarrow E_{\eta }\left\{
m_{1}\left( O,h\right) \right\} $ is linear and continuous, $b\left(
Z\right) =-{\mathcal{R}_{1}\left( Z\right) }/{E_{\eta }\left(
S_{ab}|Z\right) }$ where $\mathcal{R}_{1}\left( Z\right) $ is the Riesz
representer of the map. The representer $\mathcal{R}_{1}\left( Z\right) $
can be many different functionals of the marginal law of $Z,$ depending on
the map it represents. The examples in the next section illustrate this
point.

\section{Examples}

\label{ch:examples}

In this section we present several examples of parameters satisfying the
mixed bias property. These examples demonstrate that the class of parameters
with the mixed bias property strictly includes the classes of \cite%
{NeweyCherno19} and of \cite{Robins08Higher} and that neither of this
classes is included in the other. In the Supplementary Material we provide
further examples.

In the following examples the parameters are, possibly
some function of, parameters that are in both the class of \cite%
{NeweyCherno19} and of \cite{Robins08Higher}

\begin{example}[Mean of an outcome that is missing at random and average
treatment effect]
\label{ex:MAR} Suppose $O=\left( DY,D,Z\right) $ where $D$ is binary, $Y$ is
an outcome which is observed if and only if $D=1$ and $Z$ is a vector of
always observed covariates. If we make the untestable assumption that the
density $p\left( y|D=0,Z\right) $ is equal to the density $p\left(
y|D=1,Z\right) ,$ i.e. that the outcome $Y$ is missing at random then, for $%
P=P_{\eta }$, the mean of $Y$ is equal to $\chi \left( \eta \right) =E_{\eta
}\left\{ a\left( Z\right) \right\} $ where $a\left( Z\right) \equiv E_{\eta
}\left( DY|Z\right) /E_{\eta }\left( D|Z\right) $. If $a\left( Z\right) \in
L_{2}\left( P_{\eta ,Z}\right) $ and $E_{\eta }\left( D|Z\right) >0,$ then
the parameter $\chi \left( \eta \right) $ satisfies the conditions of
Proposition \ref{prop:3} with $m_{1}\left( O,h\right) \equiv h$ and $S_{0}=0.
$ The map $h\in L_{2}\left( P_{\eta ,Z}\right) \rightarrow E_{\eta }\left\{
m_{1}\left( O,h\right) \right\} $ is continuous with Riesz representer $%
\mathcal{R}_{1}\left( Z\right) =1,$ and $a\left( Z\right) =-E_{\eta }\left\{
q\left( O\right) |Z\right\} /E_{\eta }\left( S_{ab}|Z\right) $ for $S_{ab}=-D
$ and $q\left( O\right) =DY.$ Consequently, $\chi \left( \eta \right) $ has
the mixed bias property for $a\left( Z\right) $ as defined and $b\left(
Z\right) =1/E_{\eta }\left( D|Z\right) .$ Since $m_{1}\left( O,a\right) =a$,
Proposition \ref{prop:2} implies that $\chi \left( \eta \right) $ is in the
class of parameters considered by \cite{Robins08Higher}. Interestingly, as
shown in \cite{NeweyCherno19} and anticipated in the previous section, the
parameter $\chi \left( \eta \right) $ is also in the class of \cite%
{NeweyCherno19}, but for a different `covariate' $Z^{\ast }.$ Specifically,
let $Z^{\ast }\equiv $ $\left( D,Z\right) $ and $a^{\ast }\left( Z^{\ast
}\right) \equiv E_{\eta }\left( DY|Z^{\ast }\right) .$ Then, we can
re-express $\chi \left( \eta \right) $ as $\chi \left( \eta \right) =E_{\eta
}\left\{ m_{1}^{\ast }\left( O,a^{\ast }\right) \right\} $ where for any $%
h^{\ast }\left( D,Z\right) ,$ $m_{1}^{\ast }\left( O,h^{\ast }\right) \equiv
h^{\ast }\left( D=1,Z\right) .$ The map $h^{\ast }\in L_{2}\left( P_{\eta
,\left( D,Z\right) }\right) \rightarrow E_{\eta }\left\{ m_{1}^{\ast }\left(
O,h^{\ast }\right) \right\} $ is linear and it is continuous when $E_{\eta
}\left\{ P_{\eta }\left( D=1|Z\right) ^{-1}\right\} <\infty $ and has Riesz
representer $\mathcal{R}_{1}^{\ast }\left( Z^{\ast }\right) =D/E_{\eta
}\left( D|Z^{\ast }\right) .$ Thus, under the latter condition, the
parameter falls in the class of \cite{NeweyCherno19}. Because $a^{\ast
}\left( Z^{\ast }\right) $ is a conditional expectation given $Z^{\ast },$
Proposition \ref{prop:3} implies that $\chi \left( \eta \right) $ has the
mixed bias property for $a^{\ast }\left( Z^{\ast }\right) $ as defined, $%
S_{ab}^{\ast }=1$ and $b^{\ast }\left( Z^{\ast }\right) =D/E_{\eta }\left(
D|Z\right) .$ In the Supplementary Web Appendix we argue that this example
implies that the average treatment effect contrast is a difference of two
parameters, each belonging to both the class of \cite{Robins08Higher} and of 
\cite{NeweyCherno19}.
\end{example}

\begin{example}[Expected conditional covariance]
\label{ex:CC} Let $O=\left( Y,D,Z\right) ,$ where $Y$ and $D$ are real
valued. Let $\chi \left( \eta \right) \equiv E_{\eta }\left\{ cov_{\eta
}\left( D,Y|Z\right) \right\} $ be the expected conditional covariance
between $D$ and $Y$. When $D$ is a binary treatment, $\chi \left( \eta
\right) $ is an important component of the variance weighted average
treatment effect \cite{Robins08Higher}. We can re-write $\chi \left( \eta
\right) =E_{\eta }\left( DY\right) +E_{\eta }\left\{ m_{1}\left( O,a\right)
\right\} $ where $m_{1}\left( O,h\right) \equiv -Dh,$ $a\left( Z\right)
\equiv E_{\eta }\left( Y|Z\right) =-$ $E_{\eta }\left\{ q\left( O\right)
|Z\right\} /E\left( S_{ab}|Z\right) $ with $q\left( O\right) =Y$ and $%
S_{ab}=-1.$ Then $\chi \left( \eta \right) $ has the mixed bias property
with $a\left( Z\right) $ as defined and $b\left( Z\right) =\mathcal{R}%
_{1}\left( Z\right) =-E_{\eta }\left( D|L\right) $ the Riesz representer of
the map $h\in L_{2}\left( P_{\eta ,Z}\right) \rightarrow E_{\eta }\left\{
m_{1}\left( O,h\right) \right\} $. Thus, $\chi \left( \eta \right) $ is in 
\cite{NeweyCherno19} and in the \cite{Robins08Higher} classes with $S_{a}=-D,
$ $S_{b}=Y$ and $S_{0}=DY.$
\end{example}

The next example gives a parameter that is in the class of \cite%
{Robins08Higher} but not in the class of \cite{NeweyCherno19}.

\begin{example}[Mean of an outcome missing not at random]
\label{ex:MNAR}Suppose $O=\left( DY,D,Z\right) $ where $D$ is binary, $Y
$ is an outcome which is observed if and only if $D=1$ and $Z$ is a vector
of always observed covariates. If we make the untestable assumption that the
density $p\left( y|D=0,Z\right) $ is a known exponential tilt of the density 
$p\left( y|D=1,Z\right) ,$ i.e. 
\begin{equation}
p\left( y|D=0,Z\right) =p\left( y|D=1,Z\right) \exp \left( \delta y\right)
/E\left\{ \exp \left( \delta Y\right) |D=1,Z\right\}   \label{tilt}
\end{equation}%
where $\delta $ is a given constant, then under $P=P_{\eta }$ the mean of $Y$
is $\chi \left( \eta \right) =E_{\eta }\left\{ DY+\left( 1-D\right) a\left(
Z\right) \right\} $ assuming $a\left( Z\right) \equiv E_{\eta }\left\{
DY\exp \left( \delta Y\right) |Z\right\} /E_{\eta }\left\{ D\exp \left(
\delta Y\right) |Z\right\} $ exists. Estimation of $\chi \left( \eta \right) 
$ under different fixed values of $\delta $ has been proposed in the
literature as a way of conducting sensitivity analysis to departures from
the missing at random assumption \citep{scharfstein}. Under the sole
restriction $\left( \ref{tilt}\right) $ the law $P$ of the observed data $O$
is unrestricted, and hence the model for $P$ is non-parametric. If $a\left(
Z\right) \in L_{2}\left( P_{\eta ,Z}\right) $ and $E_{\eta }\left\{ D\exp
\left( \delta Y\right) |Z\right\} >0,$ then the parameter $\psi \left( \eta
\right) \equiv E_{\eta }\left\{ m_{1}\left( O,a\right) \right\} $ with $%
m_{1}\left( O,h\right) \equiv \left( 1-D\right) h$ has the mixed bias
property because it satisfies the conditions of Proposition \ref{prop:3}
with $q\left( O\right) =DY\exp \left( \delta Y\right) ,$ $S_{ab}=-D\exp
\left( \delta Y\right) $ and Riesz representer $\mathcal{R}_{1}\left(
Z\right) =E_{\eta }\left( 1-D|Z\right) $ and $b\left( Z\right) \equiv
-E_{\eta }\left( 1-D|Z\right) /E_{\eta }\left\{ D\exp \left( \delta Y\right)
|Z\right\} .$ Thus, $\chi \left( \eta \right) $ also satisfies the mixed
bias property with $S_{ab},$ $a$ and $b$ as defined\thinspace . The
influence function of $\chi \left( \eta \right) $ was derived in \cite%
{robinscomment} and was shown to be in the \cite{Robins08Higher} class in
that paper.  In the Appendix we show that when $\delta \not=0,$ there exists no linear and continuous map $h^{\ast }\in
L_{2}\left( P_{\eta ,\left( Z\right) }\right) \rightarrow E_{\eta }\left\{
m_{1}^{\ast }\left( O,h^{\ast }\right) \right\} ,$ such that $\psi \left(
\eta \right) =E_{\eta }\left\{ m_{1}^{\ast }\left( O,a^{\ast }\right)
\right\} $ for $a^{\ast }\left( Z\right) =E_{\eta }\left\{ q\left( O\right)
|Z\right\} $ and $q\left( O\right) $ some statistic. 
We also show that there exists no linear and continuous map $h^{\ast }\in
L_{2}\left( P_{\eta ,\left(D, Z\right) }\right) \rightarrow E_{\eta }\left\{
m_{1}^{\ast }\left( O,h^{\ast }\right) \right\} ,$ such that $\psi \left(
\eta \right) =E_{\eta }\left\{ m_{1}^{\ast }\left( O,a^{\ast }\right)
\right\} $ for $a^{\ast }\left(D, Z\right) =E_{\eta }\left\{ q\left( O\right)
|D,Z\right\} $ and $q\left( O\right) $ some statistic. 
This shows that $\psi
\left( \eta \right) ,$ and consequently $\chi \left( \eta \right) ,$ is not
in the class studied in \cite{NeweyCherno19}.
\end{example}

The next example gives a parameter that is in the class of \cite%
{NeweyCherno19} but not in the class of \cite{Robins08Higher}

\begin{example}[Causal effect of a treatment taking values on a continuum]
\label{ex:APE} Let $O=\left( Y,D,L\right) ,$ $Z=\left( D,L\right) ,$where $Y$
and $D$ are real valued, $D$ is a treatment variable taking any value in $%
\left[ 0,1\right] $ and $L$ is a covariate vector. Furthermore, let $Y_{d}$
\ denote the counterfactual outcome under treatment $D=d.$ Assume that $%
E_{\eta }\left( Y_{d}|L\right) =E_{\eta }\left( Y|D=d,L\right) $. The
parameter $\chi \left( \eta \right) \equiv E_{\eta }\left\{ m_{1}\left(
O,a\right) \right\} $ with $a\left( D,L\right) \equiv E_{\eta }\left(
Y|D,L\right) ,m_{1}\left( O,a\right) \equiv \int_{0}^{1}a\left( u,L\right)
w\left( u\right) du$ where $w\left( \cdot \right) $ is a given scalar
function satisfying $\int_{0}^{1}w\left( u\right) du=0$ agrees with the
treatment effect contrast $\int_{0}^{1}E_{\eta }\left( Y_{u}\right) w\left(
u\right) du.$ The map $h\in L_{2}\left( P_{\eta ,\left( D,Z\right) }\right)
\rightarrow E_{\eta }\left\{ m_{1}\left( O,h\right) \right\} $ where $%
m_{1}\left( O,h\right) \equiv \int_{0}^{1}h\left( u,L\right) w\left(
u\right) du$ is continuous if $E_{\eta }\left\{ \left\{ w\left( D\right)
/f\left( D|L\right) \right\} ^{2}\right\} <\infty $ with Riesz representer $%
\mathcal{R}_{1}\left( Z\right) =w\left( D\right) /f\left( D|L\right) .$ In
such case, the parameter $\chi \left( \eta \right) $ is in the class studied
in \cite{NeweyCherno19}. Thus, by Proposition \ref{prop:3} it has the mixed
bias property with $S_{ab}=-1,$ $a$ as defined, and $b\left( Z\right) =%
\mathcal{R}_{1}\left( Z\right) =w\left( D\right) /f\left( D|L\right)$. However, in the Appendix we show that $\chi \left( \eta
\right) $ is not in the class of \cite{Robins08Higher}.
\end{example}

The next example gives a parameter that is in neither the class of \cite%
{NeweyCherno19} nor in the class of \cite{Robins08Higher}

\begin{example}
\label{ex:artificial} The following toy example illustrates that there exist
parameters $\chi \left( \eta \right) $ that have the mixed bias property but
that are in neither the class of \ \cite{NeweyCherno19} nor in the class of 
\cite{Robins08Higher}. Let $O=\left( Y_{1},Y_{2},Z\right) $ for $Y_{1}$ and $%
Y_{2}$ continuous random variables, $Y_{2}>0$ and $Z$ a scalar vector taking
any values in $\left[ 0,1\right] $. The parameter $\chi \left( \eta \right)
\equiv \int_{0}^{1}a\left( z\right) dz$ where $a\left( Z\right) \equiv
E_{\eta }\left( Y_{1}|Z\right) /E_{\eta }\left( Y_{2}|Z\right) $ can be
written as $\chi \left( \eta \right) =E_{\eta }\left\{ m_{1}\left(
O,a\right) \right\} $ where for any $h\left( z\right) ,$ $m_{1}\left(
O,h\right) \equiv \int_{0}^{1}h\left( z\right) dz$ does not depend on $O.$
The map $h\in L_{2}\left( P_{\eta ,Z}\right) \rightarrow E_{\eta }\left\{
m_{1}\left( O,h\right) \right\} $ is linear. It is continuous if $E_{\eta
}\left\{ f\left( Z\right) ^{-2}\right\} <\infty $ and has Riesz representer $%
\mathcal{R}_{1}\left( Z\right) =f\left( Z\right) ^{-1}.$ In such case, by
proposition \ref{prop:3}, $\chi \left( \eta \right) $ satisfies the mixed
bias property with $S_{ab}=-Y_{2},$ $a$ as defined and $b\left( Z\right)
=\left\{ f\left( Z\right) E_{\eta }\left( Y_{2}|Z\right) \right\} ^{-1}.$
However, it can be shown that the parameter is in neither the class studied
in \cite{NeweyCherno19} nor in the class proposed in \cite{Robins08Higher}
\end{example}

\section{Final remarks}

In \S\ $\ref{sec:setup}$ we have argued that parameters with the mixed bias
property admit estimators with the `rate double robustness' property.
However, the class of parameters with the mixed bias property does not
exhaust all parameters that admit rate doubly robust estimators. For
instance, consider $\psi \left( \eta \right) =g\left\{ \chi \left( \eta
\right) \right\} $ for a non-linear continuously differentiable function $g$
and a parameter $\chi \left( \eta \right) $ with the mixed bias property. By
Theorem \ref{theo:1}, the influence function of $\chi \left( \eta \right) $
is of the form $\left( \ref{eq:IF}\right) .$ However, the influence
function of $\psi \left( \eta \right) $ is $\psi _{\eta }^{1}=g^{\prime
}\left\{ \chi \left( \eta \right) \right\} \chi _{\eta }^{1}$ $\ $which is
not of the form $\left( \ref{eq:IF}\right) .$ Thus, by Theorem \ref{theo:2} 
$\psi \left( \eta \right) $ does not have the mixed bias property. Yet, if $%
\widetilde{\chi }$ is the rate doubly robust, cross-fitted, one step
estimator of \S\ $\ref{sec:setup},$ then by the delta method, $\widetilde{%
\psi }=g\left( \widetilde{\chi }\right) $ is a rate doubly robust estimator
of $\psi \left( \eta \right) .$ A characterization of the class of all
parameters that admit rate doubly robust estimators remains an open question.

\newpage

\section{Supplementary Material}

\subsection{Examples}
\begin{example}[Population average treatment effect]
\label{ex:ATE} Suppose that $O=\left( Y,D,Z\right) $ where $D$ is a binary
treatment indicator, $Y$ is an outcome and $Z$ is a baseline covariate
vector. Under the assumption of unconfoundedness given $Z,$ the population
average treatment effect contrast is $ATE\left( \eta \right) \equiv \chi
_{1}\left( \eta \right) -\chi _{2}\left( \eta \right) $ where $\chi
_{1}\left( \eta \right) \equiv E_{\eta }\left\{ a_{1}\left( Z\right)
\right\} $ and $\chi _{2}\left( \eta \right) \equiv E_{\eta }\left\{
a_{2}\left( Z\right) \right\} $ with $a_{1}\left( Z\right) \equiv E_{\eta
}\left( DY|Z\right) /E_{\eta }\left( D|Z\right) $ and $a_{2}\left( Z\right)
\equiv E_{\eta }\left\{ \left( 1-D\right) Y|Z\right\} /E_{\eta }\left\{
\left( 1-D\right) |Z\right\} $. Regarding $1-D\,\ $as another missing data
indicator, example $\left( \ref{ex:MAR}\right) $ implies that $ATE\left(
\eta \right) $ is a difference of two parameters, $\chi _{1}\left( \eta
\right) $ and $\chi _{2}\left( \eta \right) ,$ each in the class of \cite%
{Robins08Higher} and of \cite{NeweyCherno19}.
\end{example}

\begin{example}[Mean of outcome missing at random in the non-respondents]
\label{ex:MAR-non-respondents}With the notation and assumptions of Example %
\ref{ex:MAR}, $E_{\eta }\left\{ \left( 1-D\right) a\left( Z\right) \right\}
/E_{\eta }\left( 1-D\right) $ where again, $a\left( Z\right) \equiv E_{\eta
}\left( DY|Z\right) /E_{\eta }\left( D|Z\right) ,$ is equal to the mean of $Y
$ among the non-respondents, i.e. in the population with $D=0.$ If $a\left(
Z\right) \in L_{2}\left( P_{\eta ,Z}\right) $ and $E_{\eta }\left(
D|Z\right) >0,$ then the parameter $\chi \left( \eta \right) \equiv $ $%
E_{\eta }\left\{ \left( 1-D\right) a\left( Z\right) \right\} $ satisfies the
conditions of \ Proposition \ref{prop:3} with $m_{1}\left( O,h\right) \equiv
\left( 1-D\right) h$ and $S_{0}=0.$ The map $h\in L_{2}\left( P_{\eta
,Z}\right) \rightarrow E_{\eta }\left\{ m_{1}\left( O,h\right) \right\} $ is
continuous with Riesz representer $\mathcal{R}_{1}\left( Z\right) =E_{\eta
}\left\{ \left( 1-D\right) |Z\right\} ,$ and $a\left( Z\right) =-E_{\eta
}\left\{ q\left( O\right) |Z\right\} /E_{\eta }\left( S_{ab}|Z\right) $ for $%
S_{ab}=-D$ and $q\left( O\right) =DY.$ Consequently, $\chi \left( \eta
\right) $ has the mixed bias property for $a\left( Z\right) $ as defined and 
$b\left( Z\right) =E_{\eta }\left\{ \left( 1-D\right) |Z\right\} /E_{\eta
}\left( D|Z\right) .$ Since $m_{1}\left( O,a\right) =\left( 1-D\right) a$,
Proposition $\ref{prop:2}$ implies that $\chi \left( \eta \right) $ is in
the class of parameters considered by \cite{Robins08Higher}. As in example %
\ref{ex:MAR}, the parameter $\chi \left( \eta \right) $ is also in the class
of \cite{NeweyCherno19}, for the different `covariate' $Z^{\ast }\equiv $ $%
\left( D,Z\right) $ and $a^{\ast }\left( Z^{\ast }\right) \equiv E_{\eta
}\left( DY|Z^{\ast }\right) $ since we can re-express $\chi \left( \eta
\right) $ as $\chi \left( \eta \right) =E_{\eta }\left\{ m_{1}^{\ast }\left(
O,a^{\ast }\right) \right\} $ where for any $h^{\ast }\left( D,Z\right) ,$ $%
m_{1}^{\ast }\left( O,h^{\ast }\right) \equiv \left( 1-D\right) h^{\ast
}\left( D=1,Z\right) .$ The map $h^{\ast }\in L_{2}\left( P_{\eta ,\left(
D,Z\right) }\right) \rightarrow E_{\eta }\left\{ m_{1}^{\ast }\left(
O,h^{\ast }\right) \right\} $ is linear and it is continuous when $E_{\eta
}\left\{ P_{\eta }\left( D=1|Z\right) ^{-1}\right\} <\infty $ and has Riesz
representer $\mathcal{R}_{1}^{\ast }\left( Z^{\ast }\right) =DE_{\eta
}\left\{ \left( 1-D\right) |Z\right\} /E_{\eta }\left( D|Z\right) .$ Thus,
under the latter condition, the parameter falls in the class of \cite%
{NeweyCherno19}. Because $a^{\ast }\left( Z^{\ast }\right) $ is a
conditional expectation given $Z^{\ast },$ Proposition \ref{prop:3} implies
that $\chi \left( \eta \right) $ has the mixed bias property for $a^{\ast
}\left( Z^{\ast }\right) $ as defined, $S_{ab}^{\ast }=1$ and $b^{\ast
}\left( Z^{\ast }\right) =DE_{\eta }\left\{ \left( 1-D\right) |Z\right\}
/E_{\eta }\left( D|Z\right) .$
\end{example}

\begin{example}[Treatment effect on the treated]
\label{ex:ATT} With the notation and assumptions of Example \ref{ex:ATE} of
the main text, the parameter $ATT\left( \eta \right) \equiv E\left(
Y|D=1\right) -\chi \left( \eta \right) /E_{\eta }\left( D\right) $ where $%
\chi \left( \eta \right) \equiv E_{\eta }\left\{ Da\left( Z\right) \right\} $
and $a\left( Z\right) $ defined as $E_{\eta }\left\{ \left( 1-D\right)
Y|Z\right\} /E_{\eta }\left\{ \left( 1-D\right) |Z\right\} $ the parameter $%
ATT\left( \eta \right) $ is the average treatment effect on the treated.
Once again, regarding $1-D$ as another missing data indicator, Example \ref%
{ex:MAR-non-respondents} implies that $ATT\left( \eta \right) $ is a
continuous function of a parameter $\chi \left( \eta \right) $ in the class
of \cite{Robins08Higher} and of \cite{NeweyCherno19}, and other parameters $%
E\left( Y|D=1\right) $ and $E_{\eta }\left( D\right) $ whose estimation does
not require the estimation of high dimensional nuisance parameters
\end{example}

\begin{example}[Average policy effect of a counterfactual change of
covariate values]
Let $\chi \left( \eta \right) \equiv \psi \left( \eta \right) -E_{\eta
}\left( Y\right) $ where $\psi \left( \eta \right) =E_{\eta }\left\{ a\left(
t\left( D\right) ,L\right) \right\} $ with $a\left( D,L\right) \equiv
E_{\eta }\left( Y|D,L\right) .$ Then, with the notation and assumptions of
example \ref{ex:APE} of the main text, $\chi \left( \eta \right) $ is the
average policy effect of a counterfactual change $d\rightarrow t\left(
d\right) $ of treatment values (\cite{stock1989}). Note that $\psi \left(
\eta \right) =E_{\eta }\left\{ m_{1}\left( O,a\right) \right\} $ where for
any $h\left( D,L\right) $ $m_{1}\left( O,h\right) =h\left\{ t\left( D\right)
,L\right\} .$ The functional $h\in L_{2}\left( P_{\eta ,\left( D,Z\right)
}\right) \rightarrow E_{\eta }\left\{ m_{1}\left( O,h\right) \right\} $ is
continuous if $E_{\eta }\left[ \left\{ f_{t}\left( D|L\right) /f\left(
D|L\right) \right\} ^{2}\right] <\infty $ where $f_{t}\left( D|L\right) $ is
the density of $t\left( D\right) $ given $L.$ The Riesz representer of the
map is $\mathcal{R}_{1}\left( Z\right) =f_{t}\left( D|L\right) /f\left(
D|L\right) .$ In such case, $\psi \left( \eta \right) $ is in the class
studied in \cite{NeweyCherno19}, and thus $\chi \left( \eta \right) $ has
the mixed bias property, with with $S_{ab}=-1,$ $a\left( Z\right) $ as
defined, and $b\left( Z\right) =\mathcal{R}_{1}\left( Z\right) =f_{t}\left(
D|L\right) /f\left( D|L\right) .$ However, it can be shown that $\chi \left(
\eta \right) $ is not in the class of \cite{Robins08Higher}.
\end{example}

\subsection{Regularity conditions}

We now state the regularity conditions invoked in several of the
propositions and theorems in the main text. These are mild conditions that
are satisfied in all the examples provided in the main text and in this
Appendix.

\begin{condition}
\label{cond:R1}There exists a dense set $H_{a}$ of $L_{2}\left( P_{\eta
,Z}\right) $ such that  $H_{a}\cap 
\mathcal{A}\not=\varnothing$, and for each $\eta $ and each $h\in H_{a},$ there
exists $\varepsilon \left( \eta ,h\right) >0$ such that $a+th$ $\in \mathcal{%
A}$ if $\left\vert t\right\vert <\varepsilon \left( \eta ,h\right) $ where $%
a\left( Z\right) \equiv a\left( Z;\eta \right) .$  The same holds replacing $a$ with $b$ and $%
\mathcal{A}$ with $\mathcal{B}.$ Furthermore $E_{\eta }\left\{ \left\vert
S_{ab}b\left( Z\right) h\left( Z\right) \right\vert \right\} <\infty $ for $%
h\in H_{a}$ and $E_{\eta }\left\{ \left\vert S_{ab}a\left( Z\right) h\left(
Z\right) \right\vert \right\} <\infty $ for $h\in H_{b}.$  Moreover for all $\eta$, $E_{\eta}\left[ \left\vert S_{ab}a^{\prime} b^{\prime} \right\vert \right] < \infty$ for all $a^{\prime}\in \mathcal{A}$ and $b^{\prime}\in \mathcal{B}$.
\end{condition}

\begin{condition}
\label{cond:R2} $\chi \left( \eta \right) $ satisfies the mixed bias
property and there exists $b^{\prime }\in \mathcal{B}$ such that for all $%
\eta ,$ (i) $b^{\prime }\left( Z\right) \not=0$ a.s.$\left( P_{\eta
,Z}\right) ,$ and (ii) for the map $m_{2}$ defined in the proof of Theorem
1, $E_{\eta }\left\{ m_{2}\left( O,b%
{\acute{}}%
\right) |Z\right\} +E_{\eta }\left( S_{ab}|Z\right) b%
{\acute{}}%
\left( Z\right) a\left( Z\right) $ is in $L_{2}\left( P_{\eta ,Z}\right) $
and $m_{2}\left( O,b\right) -m_{2}\left( O,b%
{\acute{}}%
\right) /b%
{\acute{}}%
\left( Z\right) $ is in $L_{2}\left( P_{\eta }\right) $.
\end{condition}

\subsection{Proofs}

\begin{proof}[of Proposition \ref{prop:0}]
Let $\eta ^{\prime }$ be such that $a^{\prime }=a$ and $b^{\prime }=b.$
Without loss of generality consider a local variation independent
parameterization $\eta =\left( a,b,\tau \right) $ and a regular parametric
submodel $t\rightarrow \eta _{t}=\left( a_{t},b_{t},\tau _{t}\right) .$ Then,%
\begin{equation*}
\left. \frac{d}{dt}\chi \left( \eta _{t}\right) \right\vert _{t=0}=\left. 
\frac{d}{dt}\chi \left( a_{t},b,\tau \right) \right\vert _{t=0}+\left. \frac{%
d}{dt}\chi \left( a,b_{t},\tau \right) \right\vert _{t=0}+\left. \frac{d}{dt}%
\chi \left( a,b,\tau _{t}\right) \right\vert _{t=0}
\end{equation*}%
By $\left( \ref{bias1}\right) ,$%
\begin{equation*}
\chi \left( a_{t},b,\tau \right) =E_{\left( a_{t},b,\tau \right) }\left\{
\chi \left( \eta ^{\prime }\right) +\chi _{\eta ^{\prime }}^{1}\right\} 
\end{equation*}%
\begin{equation*}
\chi \left( a,b_{t},\tau \right) =E_{\left( a,b_{t},\tau \right) }\left\{
\chi \left( \eta ^{\prime }\right) +\chi _{\eta ^{\prime }}^{1}\right\} 
\end{equation*}%
and%
\begin{equation*}
\chi \left( a,b,\tau _{t}\right) =E_{\left( a,b,\tau _{t}\right) }\left\{
\chi \left( \eta ^{\prime }\right) +\chi _{\eta ^{\prime }}^{1}\right\} 
\end{equation*}%
Then, 
\begin{eqnarray*}
\left. \frac{d}{dt}\chi \left( \eta _{t}\right) \right\vert _{t=0} &=&\left. 
\frac{d}{dt}E_{\left( a_{t},b,\tau \right) }\left\{ \chi \left( \eta
^{\prime }\right) +\chi _{\eta ^{\prime }}^{1}\right\} \right\vert _{t=0} \\
&&+\left. \frac{d}{dt}E_{\left( a,b_{t},\tau \right) }\left\{ \chi \left(
\eta ^{\prime }\right) +\chi _{\eta ^{\prime }}^{1}\right\} \right\vert
_{t=0} \\
&&+\left. \frac{d}{dt}E_{\left( a,b,\tau _{t}\right) }\left\{ \chi \left(
\eta ^{\prime }\right) +\chi _{\eta ^{\prime }}^{1}\right\} \right\vert
_{t=0} \\
&=&\left. \frac{d}{dt}E_{\eta _{t}}\left\{ \chi \left( \eta ^{\prime
}\right) +\chi _{\eta ^{\prime }}^{1}\right\} \right\vert _{t=0} \\
&=&E_{\eta }\left[ \left\{ \chi \left( \eta ^{\prime }\right) +\chi _{\eta
^{\prime }}^{1}\right\} g\right] 
\end{eqnarray*}%
Consequently, 
\begin{eqnarray*}
\chi _{\eta }^{1} &=&\chi \left( \eta ^{\prime }\right) +\chi _{\eta
^{\prime }}^{1}-E_{\eta }\left\{ \chi \left( \eta ^{\prime }\right) +\chi
_{\eta ^{\prime }}^{1}\right\}  \\
&=&\chi \left( \eta ^{\prime }\right) +\chi _{\eta ^{\prime }}^{1}-\chi
\left( \eta \right) 
\end{eqnarray*}%
Thus, $\chi _{\eta }^{1}+\chi \left( \eta \right) =\chi \left( \eta ^{\prime
}\right) +\chi _{\eta ^{\prime }}^{1}$ which proves the proposition.
\end{proof}

\begin{proof}[of Theorem \protect\ref{theo:1}]
Fix $a^{\ast }\in  \mathcal{A}$ and $b^{\ast }\in \mathcal{B}$ and define 
\begin{equation*}
S_{0}^{\ast }\equiv \left( \chi +\chi ^{1}\right) _{\left( a^{\ast },b^{\ast
}\right) }-S_{ab}a^{\ast }b^{\ast }
\end{equation*}%
\begin{equation*}
m_{1}^{\ast }\left( O,a\right) \equiv \left\{ \left( \chi +\chi ^{1}\right)
_{\left( a,b^{\ast }\right) }-S_{ab}ab^{\ast }\right\} -S_{0}^{\ast }
\end{equation*}%
\begin{equation*}
m_{2}^{\ast }\left( O,b\right) \equiv \left\{ \left( \chi +\chi ^{1}\right)
_{\left( a^{\ast },b\right) }-S_{ab}a^{\ast }b\right\} -S_{0}^{\ast }
\end{equation*}%
For any $h\in \mathcal{A}$ we have%
\begin{eqnarray}
E_{\eta }\left[ m_{1}^{\ast }\left( O,h\right) \right]  &=&E_{\eta }\left\{
\left( \chi +\chi ^{1}\right) _{\left( h,b^{\ast }\right) }-S_{ab}hb^{\ast
}-S_{0}^{\ast }\right\}   \notag \\
&=&E_{\eta }\left\{ \left( \chi +\chi ^{1}\right) _{\left( h,b^{\ast
}\right) }-S_{ab}hb^{\ast }\right\} -E_{\eta }\left\{ \left( \chi +\chi
^{1}\right) _{\left( a^{\ast },b^{\ast }\right) }-S_{ab}a^{\ast }b^{\ast
}\right\}   \notag \\
&=&E_{\eta }\left\{ \left( \chi +\chi ^{1}\right) _{\left( h,b^{\ast
}\right) }-\chi \left( \eta \right) -S_{ab}hb^{\ast }\right\} -E_{\eta
}\left\{ \left( \chi +\chi ^{1}\right) _{\left( a^{\ast },b^{\ast }\right)
}-\chi \left( \eta \right) -S_{ab}a^{\ast }b^{\ast }\right\}   \notag \\
&=&E_{\eta }\left( -S_{ab}hb^{\ast }\right) +E_{\eta }\left\{ S_{ab}\left(
a-h\right) \left( b-b^{\ast }\right) \right\} -E_{\eta }\left\{ S_{ab}\left(
a-a^{\ast }\right) \left( b-b^{\ast }\right) \right\} +E_{\eta }\left(
S_{ab}a^{\ast }b^{\ast }\right)   \notag \\
&=&E_{\eta }\left\{ S_{ab}\left( a-h\right) b\right\} -E_{\eta }\left\{
S_{ab}\left( a-a^{\ast }\right) b\right\}   \notag \\
&=&-E_{\eta }\left( S_{ab}hb\right) +E_{\eta }\left( S_{ab}a^{\ast }b\right) 
\label{eq:affineproof}
\end{eqnarray}%
Likewise, by symmetry we have established that 
\begin{equation*}
E_{\eta }\left\{ m_{2}^{\ast }\left( O,h\right) \right\} =-E_{\eta }\left(
S_{ab}ha\right) +E_{\eta }\left( S_{ab}ab^{\ast }\right) 
\end{equation*}%
We will next show that 
\begin{equation}
\chi _{\eta }^{1}=S_{ab}ab+m_{1}^{\ast }\left( O,a\right) +m_{2}^{\ast
}\left( O,b\right) +S_{0}^{\ast }-\chi \left( \eta \right) 
\label{eq:IFstar}
\end{equation}%
To do so, it suffices to show that

(I) $E_{\eta }\left\{ S_{ab}ab+m_{1}^{\ast }\left( O,a\right) +m_{2}^{\ast
}\left( O,b\right) +S_{0}^{\ast }\right\} =\chi \left( \eta \right) $ and

(II) $\left. \frac{d}{dt}E_{\eta }\left\{ S_{ab}a_{t}b_{t}+m_{1}^{\ast
}\left( O,a_{t}\right) +m_{2}^{\ast }\left( O,b_{t}\right) +S_{0}\right\}
\right\vert _{t=0}=0$ for any regular submodel $t\rightarrow P_{\eta _{t}}$

To show (I) we write 
\begin{eqnarray*}
&&E_{\eta }\left\{ m_{1}^{\ast }\left( O,a\right) +m_{2}^{\ast }\left(
O,b\right) +S_{0}^{\ast }\right\}  \\
&=&E_{\eta }\left\{ \left( \chi +\chi ^{1}\right) _{\left( a,b^{\ast
}\right) }-S_{ab}ab^{\ast }+\left( \chi +\chi ^{1}\right) _{\left( a^{\ast
},b\right) }-S_{ab}a^{\ast }b-S_{0}\right\}  \\
&=&E_{\eta }\left[ \left( \chi +\chi ^{1}\right) _{\left( a,b^{\ast }\right)
}-\chi \left( \eta \right) -S_{ab}ab^{\ast }+\left( \chi +\chi ^{1}\right)
_{\left( a^{\ast },b\right) }-\chi \left( \eta \right) -S_{ab}a^{\ast
}b-\left\{ S_{0}-\chi \left( \eta \right) \right\} \right] +\chi \left( \eta
\right)  \\
&=&E_{\eta }\left( -S_{ab}ab^{\ast }-S_{ab}a^{\ast }b\right) -E_{\eta
}\left\{ S_{0}-\chi \left( \eta \right) \right\} +\chi \left( \eta \right) 
\\
&=&E_{\eta }\left( -S_{ab}ab^{\ast }-S_{ab}a^{\ast }b\right) -E_{\eta
}\left\{ \left( \chi +\chi ^{1}\right) _{\left( a^{\ast },b^{\ast }\right)
}-\chi \left( \eta \right) -S_{ab}a^{\ast }b^{\ast }\right\} +\chi \left(
\eta \right)  \\
&=&E_{\eta }\left( -S_{ab}ab^{\ast }-S_{ab}a^{\ast }b\right) -E_{\eta
}\left\{ S_{ab}\left( a-a^{\ast }\right) \left( b-b^{\ast }\right) \right\}
+E_{\eta }\left( S_{ab}a^{\ast }b^{\ast }\right) +\chi \left( \eta \right) 
\\
&=&-E_{\eta }\left( S_{ab}ab\right) +\chi \left( \eta \right) 
\end{eqnarray*}%
which shows (I)

To show (II) we note that by $\left( \ref{eq:affineproof}\right) $ 
\begin{equation*}
E_{\eta }\left\{ S_{ab}a_{t}b+m_{1}^{\ast }\left( O,a_{t}\right) \right\}
=E_{\eta }\left( S_{ab}a^{\ast }b\right) 
\end{equation*}%
and the right hand side does not depend on $a_{t}$. Likewise, 
\begin{equation*}
E_{\eta }\left\{ S_{ab}ab_{t}+m_{2}^{\ast }\left( O,b_{t}\right) \right\}
=E_{\eta }\left( S_{ab}ab^{\ast }\right) 
\end{equation*}%
Thus, 
\begin{eqnarray*}
&&\left. \frac{d}{dt}E_{\eta }\left\{ S_{ab}a_{t}b_{t}+m_{1}^{\ast }\left(
O,a_{t}\right) +m_{2}^{\ast }\left( O,b_{t}\right) +S_{0}^{\ast }\right\}
\right\vert _{t=0} \\
&=&\left. \frac{d}{dt}E_{\eta }\left\{ S_{ab}a_{t}b+m_{1}^{\ast }\left(
O,a_{t}\right) \right\} \right\vert _{t=0}+\left. \frac{d}{dt}E_{\eta
}\left\{ S_{ab}ab_{t}+m_{2}^{\ast }\left( O,b_{t}\right) \right\}
\right\vert _{t=0} \\
&=&0
\end{eqnarray*}%
This shows part (II) and thus concludes the proof of $\left( \ref{eq:IFstar}%
\right) .$

Next, take $a^{\dag }\in $ $H_{a}\cap \mathcal{A}$ and $b^{\dag }\in
H_{b}\cap \mathcal{B}$ which we know exist by Condition R.1. Also, by
Condition R.1 we know that $a^{\ast \ast }\equiv a^{\ast }+\varepsilon
a^{\dag }$ $\in \mathcal{A}$ and $b^{\ast \ast }\equiv b^{\ast }+\varepsilon
b^{\dag }$ $\in \mathcal{B}$ for an $\varepsilon >0$ sufficiently small.
Now, define $S_{0}^{\ast \ast },m_{1}^{\ast \ast }\left( O,a\right) $ and $%
m_{2}^{\ast \ast }\left( O,b\right) $ like $S_{0}^{\ast },m_{1}^{\ast
}\left( O,a\right) $ and $m_{2}^{\ast }\left( O,b\right) $ but using $%
a^{\ast \ast }$ and $b^{\ast \ast }$ instead of $a^{\ast }$ and $b^{\ast }.$
Then, 
\begin{equation*}
\chi _{\eta }^{1}=S_{ab}ab+m_{1}^{\ast \ast }\left( O,a\right) +m_{2}^{\ast
\ast }\left( O,b\right) +S_{0}^{\ast \ast }-\chi \left( \eta \right). 
\end{equation*}%
So, combining this equality with $\left( \ref{eq:IFstar}\right) $ we
conclude that 
\begin{equation*}
m_{1}^{\ast \ast }\left( O,a\right) +m_{2}^{\ast \ast }\left( O,b\right)
+S_{0}^{\ast \ast }=m_{1}^{\ast }\left( O,a\right) +m_{2}^{\ast }\left(
O,b\right) +S_{0}^{\ast }. 
\end{equation*}%
Thus, 
\begin{equation*}
m_{1}^{\ast \ast }\left( O,a\right) -m_{1}^{\ast }\left( O,a\right)
=m_{2}^{\ast }\left( O,b\right) -m_{2}^{\ast \ast }\left( O,b\right)
+S_{0}^{\ast }-S_{0}^{\ast \ast }. 
\end{equation*}%
The right hand side depends on $b\,\ $and the data $O,$ but the left hand
side depends on $a$ and the data $O$. Thus, we conclude that $m_{1}^{\ast
\ast }\left( O,a\right) -m_{1}^{\ast }\left( O,a\right) $ is a statistic $%
Q_{1}^{\dag }$ that does not depend on $\eta .$

Now, by $\left( \ref{eq:affineproof}\right) ,$ 
\begin{eqnarray*}
E_{\eta }\left( Q_{1}^{\dag }\right)  &=&E_{\eta }\left\{ m_{1}^{\ast \ast
}\left( O,a\right) \right\} -E_{\eta }\left\{ m_{1}^{\ast }\left( O,a\right)
\right\}  \\
&=&\left\{ -E_{\eta }\left( S_{ab}ab\right) +E_{\eta }\left( S_{ab}a^{\ast
\ast }b\right) \right\} -\left\{ -E_{\eta }\left( S_{ab}ab\right) +E_{\eta
}\left( S_{ab}a^{\ast }b\right) \right\}  \\
&=&E_{\eta }\left\{ S_{ab}\left( a^{\ast \ast }-a^{\ast }\right) b\right\} 
\\
&=&\varepsilon E_{\eta }\left( S_{ab}a^{\dag }b\right) 
\end{eqnarray*}%
Next, let $S_{0}^{\dag },m_{1}^{\dag }\left( O,a\right) $ and $m_{2}^{\dag
}\left( O,b\right) $ be defined like $S_{0}^{\ast },m_{1}^{\ast }\left(
O,a\right) $ and $m_{2}^{\ast }\left( O,b\right) $ but using $a^{\dag }$ and 
$b^{\dag }$ instead of $a^{\ast }$ and $b^{\ast }.$ By $\left( \ref%
{eq:affineproof}\right) \,\ $applied to $m_{1}^{\dag }\left( O,a\right) $
and $a^{\dag }$ instead of $m_{1}^{\ast }\left( O,a\right) $ and $a^{\ast }$
we have that 
\begin{eqnarray*}
E_{\eta }\left\{ m_{1}^{\dag }\left( O,h\right) \right\}  &=&-E_{\eta
}\left( S_{ab}hb\right) +E_{\eta }\left( S_{ab}a^{\dag }b\right)  \\
&=&-E_{\eta }\left( S_{ab}hb\right) +E_{\eta }\left( Q_{1}^{\dag }\right)
/\varepsilon .
\end{eqnarray*}%
Consequently, 
\begin{equation*}
m_{1}\left( O,h\right) \equiv m_{1}^{\dag }\left( O,h\right) -Q_{1}^{\dag
}/\varepsilon 
\end{equation*}%
satisfies 
\begin{equation}
E_{\eta }\left\{ m_{1}\left( O,h\right) \right\} =-E_{\eta }\left(
S_{ab}hb\right) \text{ for all }h\in \mathcal{A}.  \label{eq:m1equality}
\end{equation}%
and therefore the map $h\in \mathcal{A}\rightarrow E_{\eta }\left\{
m_{1}\left( O,h\right) \right\} $ is linear. In fact, for any $%
h_{1},h_{2}\in \mathcal{A}$ and constants $\alpha _{1}$ and $\alpha _{2}$
such that $\alpha _{1}h_{1}+\alpha _{2}h_{2}\in \mathcal{A}$, we know that 
\begin{equation*}
E_{\eta }\left\{ m_{1}\left( O,\alpha _{1}h_{1}+\alpha _{2}h_{2}\right)
\right\} =\alpha _{1}E_{\eta }\left\{ m_{1}\left( O,h_{1}\right) \right\}
+\alpha _{2}E_{\eta }\left\{ m_{1}\left( O,h_{2}\right) \right\} 
\end{equation*}%
is true for all $\eta .$ Then, for all $\eta ^{\prime }$ 
\begin{equation*}
E_{\eta ^{\prime }}\left[ m_{1}\left( O,\alpha _{1}h_{1}+\alpha
_{2}h_{2}\right) -\left\{ \alpha _{1}m_{1}\left( O,h_{1}\right) +\alpha
_{2}m_{1}\left( O,h_{2}\right) \right\} \right] =0
\end{equation*}%
By assumption the random variable $r\left( O\right) \equiv m_{1}\left(
O,\alpha _{1}h_{1}+\alpha _{2}h_{2}\right) -\left\{ \alpha _{1}m_{1}\left(
O,h_{1}\right) +\alpha _{2}m_{1}\left( O,h_{2}\right) \right\} $ is in $%
L_{2}\left( P_{\eta }\right) .$ The linearity a.s.$\left( P_{\eta }\right) $
of the map $h\in \mathcal{A}\rightarrow m_{1}\left( O,h\right) $ follows
from Lemma \ref{lemma:a1} below which implies that $r\left( O\right) =0$ a.s.$\left(
P_{\eta }\right) .$

Likewise, we can show that there exists $Q_{2}^{\dag }$ and $m_{2}\left(
O,h\right) \equiv m_{2}^{\dag }\left( O,h\right) -Q_{2}^{\dag }/\varepsilon $
such that $h\in \mathcal{B}\rightarrow E_{\eta }\left\{ m_{2}\left(
O,h\right) \right\} =-E_{\eta }\left( S_{ab}ah\right) $ and the map $h\in 
\mathcal{B\rightarrow }m_{2}\left( O,h\right) $ is linear. Finally, define $%
S_{0}=S_{0}^{\dag }+Q_{1}^{\dag }/\varepsilon +Q_{2}^{\dag }/\varepsilon $
and conclude from $\chi _{\eta }^{1}=S_{ab}ab+m_{1}^{\dag }\left( O,a\right)
+m_{2}^{\dag }\left( O,b\right) +S_{0}^{\dag }-\chi \left( \eta \right) $
that 
\begin{equation*}
\chi _{\eta }^{1}=S_{ab}ab+m_{1}\left( O,a\right) +m_{2}\left( O,b\right)
+S_{0}-\chi \left( \eta \right) .
\end{equation*}

In addition, from $\left( \ref{eq:m1equality}\right) $ and its analogous for 
$b,$ we have that 
\begin{equation*}
E_{\eta }\left\{ S_{ab}ab+m_{1}\left( O,a\right) \right\} =0
\end{equation*}%
and 
\begin{equation*}
E_{\eta }\left\{ S_{ab}ab+m_{2}\left( O,b\right) \right\} =0.
\end{equation*}%
Consequently, 
\begin{eqnarray*}
\chi \left( \eta \right)  &=&E_{\eta }\left\{ m_{1}\left( O,a\right)
\right\} +E_{\eta }\left( S_{0}\right)  \\
&=&E_{\eta }\left\{ m_{2}\left( O,b\right) \right\} +E_{\eta }\left(
S_{0}\right) 
\end{eqnarray*}%
thus showing $\left( \ref{eq:parameter}\right) $ holds. This concludes the
proof of the Theorem.
\end{proof}

\begin{lemma}\label{lemma:a1}
Suppose that $r\left( O\right) $ is in $L_{2}\left( P_{\eta }\right) $ and
that for all $\eta 
{\acute{}}%
,$ $E_{\eta 
{\acute{}}%
}\left\{ r\left( O\right) \right\} =0.$ Then $r\left( O\right) =0$ a.s.$%
\left( P_{\eta }\right) .$
\end{lemma}

\begin{proof}[of Lemma \ref{lemma:a1}]
Suppose first that $r\left( O\right) $ is bounded. Then consider the
submodel $t\rightarrow p_{t}\left( O\right) =p_{\eta }\left( O\right)
\left\{ 1+tr\left( O\right) \right\} $ . Note that for $t$ sufficiently
small, $p_{t}>0$ (by the boundedness of $r\left( O\right) $) and $p_{t}$
integrates to 1 because $E_{\eta 
{\acute{}}%
}\left\{ r\left( O\right) \right\} =0.$ Then score of the submodel is $%
r\left( O\right) .$ Then, since by assumption the mean $E_{t}\left\{ r\left(
O\right) \right\} $ of $r\left( O\right) $ under $p_{t}$ satisfies $%
E_{p_{t}}\left\{ r\left( O\right) \right\} =0$ for all $t,$ we have 
\begin{equation*}
0=\left. \frac{d}{dt}E_{t}\left\{ r\left( O\right) \right\} \right\vert
_{t=0}=E_{\eta }\left\{ r\left( O\right) ^{2}\right\} 
\end{equation*}%
Consequently $r\left( O\right) =0$ a.s.$\left( P_{\eta }\right) .$ Next,
given an arbitrary $r\left( O\right) $ in $L_{2}\left( P_{\eta }\right) $
such that $E_{\eta 
{\acute{}}%
}\left\{ r\left( O\right) \right\} =0$ for all $\eta 
{\acute{}}%
,$ consider define $r_{n}\left( O\right) \equiv r\left( O\right) I_{\left(
-n,n\right) }\left\{ r\left( O\right) \right\} -E_{\eta }\left\{ r\left(
O\right) I_{\left( -n,n\right) }\left( r\left( O\right) \right) \right\} .$
Then $r_{n}\left( O\right) $ satisfies $E_{\eta }\left\{ r_{n}\left(
O\right) \right\} =0$ and is bounded. So, $r_{n}\left( O\right) =0$ a.s.$%
\left( P_{\eta }\right) .$ However, $r_{n}\left( O\right) $ converges in $%
L_{2}\left( P_{\eta }\right) $ to $r\left( O\right) $ so $r\left( O\right) =0
$ a.s.$\left( P_{\eta }\right) .$
\end{proof}

\begin{proof}[of Theorem \protect\ref{theo:2}]
For any fixed $h\in H_{a},$ and a given $\eta =\left( a,b,\tau \right) $
consider a parametric submodel $t\rightarrow P_{\eta _{t}}$ where $\eta
_{t}=\left( a_{t},b,\tau \right) $ with $a_{t}=a+th$ and $\left\vert
t\right\vert <\varepsilon \left( \eta ,h\right) $ as in Condition \ref%
{cond:R1}. Then, since $\chi _{\eta }^{1}$ is an influence function of the
form $\left( \ref{eq:IF}\right) $ we have 
\begin{eqnarray*}
0 &=&\frac{d}{dt}\left. E_{\eta }\left\{ S_{ab}a_{t}b+m_{1}\left(
O,a_{t}\right) +m_{2}\left( O,b\right) +S_{0}\right\} \right\vert _{t=0} \\
&=&\frac{d}{dt}\left. E_{\eta }\left\{ S_{ab}\left( a+th\right)
b+m_{1}\left( O,a\right) +tm_{1}\left( O,h\right) \right\} \right\vert _{t=0}
\\
&=&E_{\eta }\left\{ S_{ab}hb+m_{1}\left( O,h\right) \right\} .
\end{eqnarray*}%
The continuity of the maps $h\in L_{2}\left( P_{\eta ,Z}\right) \rightarrow
E_{\eta }\left\{ m_{1}\left( O,h\right) \right\} $ and $h\in L_{2}\left(
P_{\eta ,Z}\right) \rightarrow E_{\eta }\left\{ S_{ab}h\left( Z\right)
b\left( Z\right) \right\} $ implies that the map $h\in L_{2}\left( P_{\eta
,Z}\right) \rightarrow E_{\eta }\left\{ S_{ab}hb+m_{1}\left( O,h\right)
\right\} $ is continuous. Then, since we have just shown that this map
evaluates to 0 at a dense set of $L_{2}\left( P_{\eta ,Z}\right) ,$ it must
equal to 0 for all $h\in L_{2}\left( P_{\eta ,Z}\right) .$ Reasoning
analogously, we arrive at the conclusion that 
\begin{equation}
E_{\eta }\left\{ S_{ab}ha+m_{2}\left( O,h\right) \right\} =0
\label{eq:unbeq1}
\end{equation}%
for all $h\in L_{2}\left( P_{\eta ,Z}\right) ,$ thus showing part (ii) of
the Theorem. Next, suppose that $a^{\prime }-a\in L_{2}\left( P_{\eta
,Z}\right) $ and $b^{\prime }-b\in L_{2}\left( P_{\eta ,Z}\right) .$ Then,
applying part (ii) we have that $E_{\eta }\left[ S_{ab}\left( a^{\prime
}-a\right) b+m_{1}\left\{ O,\left( a^{\prime }-a\right) \right\} \right] =0$
and $E_{\eta }\left[ S_{ab}\left( b^{\prime }-b\right) a+m_{2}\left\{
O,\left( b^{\prime }-b\right) \right\} \right] =0$. Consequently, 
\begin{eqnarray*}
E_{\eta }\left\{ \chi \left( \eta ^{\prime }\right) +\chi _{\eta ^{\prime
}}^{1}\right\} -\chi \left( \eta \right)  &=&E_{\eta }\left\{ \chi \left(
\eta ^{\prime }\right) +\chi _{\eta ^{\prime }}^{1}\right\} -E_{\eta
}\left\{ \chi \left( \eta \right) +\chi _{\eta }^{1}\right\}  \\
&=&E_{\eta }\left\{ S_{ab}a^{\prime }b^{\prime }+m_{1}\left( O,a^{\prime
}\right) +m_{2}\left( O,b^{\prime }\right) \right\} - \\
&&E_{\eta }\left\{ S_{ab}ab+m_{1}\left( O,a\right) +m_{2}\left( O,b\right)
\right\}  \\
&=&E_{\eta }\left\{ S_{ab}a^{\prime }b^{\prime }+m_{1}\left( O,a^{\prime
}\right) +m_{2}\left( O,b^{\prime }\right) \right\}  \\
&&-E_{\eta }\left\{ S_{ab}ab+m_{1}\left( O,a\right) +m_{2}\left( O,b\right)
\right\}  \\
&&-E_{\eta }\left\{ S_{ab}\left( a^{\prime }-a\right) b+m_{1}\left( O,\left(
a^{\prime }-a\right) \right) \right\}  \\
&&-E_{\eta }\left\{ S_{ab}\left( b^{\prime }-b\right) a+m_{2}\left( O,\left(
b^{\prime }-b\right) \right) \right\}  \\
&=&E_{\eta }\left\{ S_{ab}\left( a-a^{\prime }\right) \left( b-b^{\prime
}\right) \right\} 
\end{eqnarray*}

thus showing part (i) of the Theorem.

Turn now to the proof of part (iii). Equation $\left( \ref{eq:unbeq1}\right) 
$ implies that for all $h\in L_{2}\left( P_{\eta ,Z}\right) ,$ 
\begin{equation*}
0=E_{\eta }\left( S_{ab}ha+\mathcal{R}_{2}h\right). 
\end{equation*}%
Thus, if $h^{\ast }\left( Z\right) \equiv E_{\eta }\left( S_{ab}|Z\right)
a\left( Z\right) +\mathcal{R}_{2}\left( Z\right) $ is in $L_{2}\left(
P_{\eta ,Z}\right) ,$ then specializing at $h=h^{\ast }$ the preceding
identity we conclude that a.s.$\left( P_{\eta ,Z}\right) $ 
\begin{equation*}
E_{\eta }\left( S_{ab}a+\mathcal{R}_{2}|Z\right) =0 
\end{equation*}%
or equivalently $a\left( Z\right) =-\mathcal{R}_{2}\left( Z\right) /E_{\eta
}\left( S_{ab}|Z\right) $ if $E_{\eta }\left( S_{ab}|Z\right) \not=0$ a.s.$%
\left( P_{\eta ,Z}\right) $. The assertion for \thinspace $b\left( Z\right) $
is proved analogously.

Next, we prove part (iv). If $b$ is in $L_{2}\left( P_{\eta ,Z}\right) ,$
then specializing $\left( \ref{eq:unbeq1}\right) $ at $h=b$ we obtain 
\begin{eqnarray}
\chi \left( \eta \right)  &=&E_{\eta }\left\{ S_{ab}ab+m_{1}\left(
O,a\right) +m_{2}\left( O,b\right) +S_{0}\right\}   \label{eq:form1} \\
&=&E_{\eta }\left\{ m_{1}\left( O,a\right) +S_{0}\right\} .  \notag
\end{eqnarray}%
On the other hand, if $b\notin L_{2}\left( P_{\eta ,Z}\right) $ but $\left(
1+t\right) b\in \mathcal{B}$ for $0<t<\varepsilon $ then, given $\eta
=\left( a,b,\tau \right) $ consider the parametric submodel $t\rightarrow
P_{\eta _{t}}$ where $\eta _{t}=\left( a,b_{t},\tau \right) $ with $%
b_{t}=b+tb$ and $0<t<\varepsilon .$ Then, by $\chi _{\eta }^{1}$ of the form 
$\left( \ref{eq:IF}\right) $ being an influence function and with $\frac{d}{%
dt^{+}}$ denoting the right derivative, we have 
\begin{eqnarray*}
0 &=&\frac{d}{dt^{+}}\left. E_{\eta }\left\{ S_{ab}ab_{t}+m_{1}\left(
O,a\right) +m_{2}\left( O,b_{t}\right) +S_{0}\right\} \right\vert _{t=0} \\
&=&\frac{d}{dt^{+}}\left. E_{\eta }\left\{ S_{ab}a\left( b+tb\right)
+m_{2}\left( O,b\right) +tm_{2}\left( O,b\right) \right\} \right\vert _{t=0}
\\
&=&E_{\eta }\left\{ S_{ab}ab+m_{2}\left( O,b\right) \right\} .
\end{eqnarray*}%
So, applying again $\left( \ref{eq:form1}\right) $ we arrive at $\chi \left(
\eta \right) =E_{\eta }\left\{ m_{1}\left( O,a\right) +S_{0}\right\} .$ The
same reasoning, but now taking left derivatives, yields to the same
conclusion if $\left( 1+t\right) b\in \mathcal{B}$ for $-\varepsilon <t<0.$
This shows (iv.b). Part (iv.a) is proved analogously. Finally, part (iv.c)
follows from 
\begin{eqnarray*}
\chi \left( \eta \right)  &=&E_{\eta }\left\{ S_{ab}ab+m_{1}\left(
O,a\right) +m_{2}\left( O,b\right) +S_{0}\right\}  \\
&=&E_{\eta }\left\{ S_{ab}ab+m_{1}\left( O,a\right) +m_{2}\left( O,b\right)
+S_{ab}ab+S_{0}-S_{ab}ab\right\}  \\
&=&E_{\eta }\left\{ S_{0}-S_{ab}ab\right\} .
\end{eqnarray*}

Turn now to the proof of part (v). By part (iii) we have that a.s.$\left(
P_{\eta ,Z}\right) $%
\begin{equation*}
a\left( Z\right) =-\frac{\mathcal{R}_{2}\left( Z\right) }{E_{\eta }\left(
S_{ab}|Z\right) }\text{ and }b\left( Z\right) =-\frac{\mathcal{R}_{1}\left(
Z\right) }{E_{\eta }\left( S_{ab}|Z\right) }
\end{equation*}%
Next, for any $h\in $ $L_{2}\left( P_{\eta ,Z}\right) ,$ write 
\begin{eqnarray*}
&&E_{\eta }\left\{ S_{ab}\frac{h^{2}}{2}+m_{1}\left( O,h\right) \right\}  \\
&=&E_{\eta }\left\{ E_{\eta }\left( S_{ab}|Z\right) \frac{h\left( Z\right)
^{2}}{2}+\mathcal{R}_{1}\left( Z\right) h\left( Z\right) \right\}  \\
&=&E_{\eta }\left[ \frac{E_{\eta }\left( S_{ab}|Z\right) }{2}\left[ h\left(
Z\right) ^{2}+2\frac{\mathcal{R}_{1}\left( Z\right) }{E_{\eta }\left(
S_{ab}|Z\right) }h\left( Z\right) +\left\{ \frac{\mathcal{R}_{1}\left(
Z\right) }{E_{\eta }\left( S_{ab}|Z\right) }\right\} ^{2}\right] \right]
-E_{\eta }\left[ \frac{E_{\eta }\left( S_{ab}|Z\right) }{2}\left\{ \frac{%
\mathcal{R}_{1}\left( Z\right) }{E_{\eta }\left( S_{ab}|Z\right) }\right\}
^{2}\right]  \\
&=&E_{\eta }\left[ \frac{E_{\eta }\left( S_{ab}|Z\right) }{2}\left\{ h\left(
Z\right) +\frac{\mathcal{R}_{1}\left( Z\right) }{E_{\eta }\left(
S_{ab}|Z\right) }\right\} ^{2}\right] -E_{\eta }\left[ \frac{E_{\eta }\left(
S_{ab}|Z\right) }{2}\left\{ \frac{\mathcal{R}_{1}\left( Z\right) }{E_{\eta
}\left( S_{ab}|Z\right) }\right\} ^{2}\right] .
\end{eqnarray*}%
So%
\begin{eqnarray*}
\arg \min_{h\in L_{2}\left( P_{\eta ,Z}\right) }E_{\eta }\left\{ S_{ab}\frac{%
h^{2}}{2}+m_{1}\left( O,h\right) \right\}  &=&\arg \min_{h\in L_{2}\left(
P_{\eta ,Z}\right) }E_{\eta }\left\{ \frac{E_{\eta }\left( S_{ab}|Z\right) }{%
2}\left[ h\left( Z\right) +\frac{\mathcal{R}_{1}\left( Z\right) }{E_{\eta
}\left( S_{ab}|Z\right) }\right] ^{2}\right\}  \\
&=&-\frac{\mathcal{R}_{1}\left( Z\right) }{E_{\eta }\left( S_{ab}|Z\right) }
\\
&=&b\left( Z\right) .
\end{eqnarray*}%
The assertion for the minimization leading to $a\left( Z\right) $ is proved
analogously.
\end{proof}

\begin{proof}[of Proposition \protect\ref{prop:2}]
With the definition of $m_{2}$ given in the proof of Theorem 1 we have that $%
E_{\eta }\left\{ m_{2}\left( O,h\right) \right\} =-E_{\eta }\left\{
S_{ab}ha\right\} $ for all $h\in \mathcal{B}$. Fix $b^{\prime }\in \mathcal{B%
}$ such that $b^{\prime }\left( Z\right) \not=0$ a.s.$\left( P_{\eta
,Z}\right) .$ Then, 
\begin{equation*}
E_{\eta 
{\acute{}}%
_{1}}\left[ E_{\eta _{2}}\left\{ m_{2}\left( O,b^{\prime }\right) |Z\right\}
+E_{\eta _{2}}\left( S_{ab}|Z\right) b^{\prime }\left( Z\right) a\left(
Z\right) \right] =0\text{ for all }\eta 
{\acute{}}%
_{1}.
\end{equation*}%
Since by assumption $a$ does not depend on $\eta 
{\acute{}}%
_{1},$ then $s_{\eta _{2}}\left( Z\right) \equiv E_{\eta _{2}}\left\{
m_{2}\left( O,b^{\prime }\right) |Z\right\} +E_{\eta _{2}}\left(
S_{ab}|Z\right) b^{\prime }\left( Z\right) a\left( Z\right) $ is a fixed
function of $Z$ (i.e. independent of $\eta 
{\acute{}}%
_{1})$ with mean zero under any marginal law of $Z.$ Hence, since by
condition 2, $s_{\eta _{2}}\left( Z\right) $ is in $L_{2}\left( P_{\eta
,Z}\right) ,$ then by Lemma \ref{lemma:a1}, $s_{\eta _{2}}\left( Z\right) =0$ a.s.$%
\left( P_{\eta ,Z}\right) ,$ from where we conclude that $a\left( Z\right)
=-E_{\eta _{2}}\left\{ q\left( O\right) |Z\right\} /E_{\eta _{2}}\left(
S_{ab}|Z\right) $ for $q\left( O\right) \equiv m_{2}\left( O,b^{\prime
}\right) |/b^{\prime }\left( Z\right) .$

Next, write for any $\eta $%
\begin{eqnarray*}
0 &=&E_{\eta }\left\{ S_{ab}ab+m_{2}\left( O,b\right) \right\}  \\
&=&E_{\eta }\left[ S_{ab}\left\{ \frac{-E_{\eta }\left[ q\left( O\right) |Z%
\right] }{E_{\eta }\left[ S_{ab}|Z\right] }\right\} b+m_{2}\left( O,b\right) %
\right]  \\
&=&E_{\eta }\left\{ -q\left( O\right) b+m_{2}\left( O,b\right) \right\} 
\end{eqnarray*}%
and since in the last display $-q\left( O\right) b+m_{2}\left( O,b\right) $
is a statistic independent of $\eta ,$ which, by condition 2, is in $%
L_{2}\left( P_{\eta ,Z}\right) $ and the display holds for all $\eta $ then $%
-q\left( O\right) b+m_{2}\left( O,b\right) =0$ a.s.$\left( P_{\eta }\right) $
for all $\eta .$ This shows that $m_{2}\left( O,h\right) =q\left( O\right) h.
$
\end{proof}

\begin{proof}[of Proposition \protect\ref{prop:3}]
For a regular parametric submodel $t\rightarrow \eta _{t}$ with score $g$ at 
$t=0$ (with $\eta _{t=0}=\eta )$, 
\begin{eqnarray*}
\left. \frac{d}{dt}\chi \left( \eta _{t}\right) \right\vert _{t=0} &=&\left. 
\frac{d}{dt}E_{\eta _{t}}\left\{ m_{1}\left( O,a_{t}\right) \right\}
+E_{\eta _{t}}\left( S_{0}\right) \right\vert _{t=0} \\
&=&E_{\eta }\left\{ m_{1}\left( O,a\right) g\right\} +\left. \frac{d}{dt}%
E_{\eta }\left\{ m_{1}\left( O,a_{t}\right) \right\} \right\vert
_{t=0}+E_{\eta }\left( S_{0}g\right) .
\end{eqnarray*}%
But, 
\begin{eqnarray*}
&&\left. \frac{d}{dt}E_{\eta }\left\{ m_{1}\left( O,a_{t}\right) \right\}
\right\vert _{t=0} \\
&=&\left. \frac{d}{dt}E_{\eta }\left\{ \mathcal{R}_{1}\left( Z\right)
a_{t}\right\} \right\vert _{t=0} \\
&=&-E_{\eta }\left[ \mathcal{R}_{1}\left( Z\right) \left. \frac{d}{dt}%
\left\{ \frac{E_{\eta _{t}}\left[ q\left( O\right) |Z\right] }{E_{\eta _{t}}%
\left[ S_{ab}|Z\right] }\right\} \right\vert _{t=0}\right]  \\
&=&-E_{\eta }\left[ \mathcal{R}_{1}\left( Z\right) \left[ \frac{E_{\eta }%
\left[ \left\{ q\left( O\right) -E_{\eta }\left\{ q\left( O\right)
|Z\right\} \right\} g|Z\right] }{E_{\eta }\left( S_{ab}|Z\right) }-E_{\eta
}\left\{ q\left( O\right) |Z\right\} \frac{E_{\eta }\left[ \left\{
S_{ab}-E_{\eta }\left( S_{ab}|Z\right) \right\} g|Z\right] }{E_{\eta }\left(
S_{ab}|Z\right) ^{2}}\right] \right]  \\
&=&E_{\eta }\left[ -\frac{\mathcal{R}_{1}\left( Z\right) }{E_{\eta }\left(
S_{ab}|Z\right) }\left[ E_{\eta }\left\{ q\left( O\right) -E_{\eta }\left\{
q\left( O\right) |Z\right\} \right\} -E_{\eta }\left[ q\left( O\right) |Z%
\right] \frac{E_{\eta }\left\{ S_{ab}-E_{\eta }\left( S_{ab}|Z\right)
\right\} }{E_{\eta }\left( S_{ab}|Z\right) }\right] g\right]  \\
&=&E_{\eta }\left[ -\frac{\mathcal{R}_{1}\left( Z\right) }{E_{\eta }\left(
S_{ab}|Z\right) }\left\{ q\left( O\right) -\frac{E_{\eta }\left\{ q\left(
O\right) |Z\right\} }{E_{\eta }\left( S_{ab}|Z\right) }S_{ab}\right\} g%
\right]  \\
&=&E_{\eta }\left[ b\left( Z\right) \left\{ q\left( O\right) +a\left(
Z\right) S_{ab}\right\} g\right] 
\end{eqnarray*}%
where 
\begin{equation*}
b\left( Z\right) \equiv -\frac{\mathcal{R}_{1}\left( Z\right) }{E_{\eta
}\left( S_{ab}|Z\right) }.
\end{equation*}%
Thus, 
\begin{eqnarray*}
\chi _{\eta }^{1} &=&m_{1}\left( O,a\right) +b\left( Z\right) \left\{
q\left( O\right) +a\left( Z\right) S_{ab}\right\} +S_{0} \\
&&-E_{\eta }\left[ m_{1}\left( O,a\right) +b\left( Z\right) \left\{ q\left(
O\right) +a\left( Z\right) S_{ab}\right\} +S_{0}\right]  \\
&=&S_{ab}ab+m_{1}\left( O,a\right) +q\left( O\right) b+S_{0} \\
&&-\chi \left( \eta \right) -E_{\eta }\left[ b\left( Z\right) \left\{
q\left( O\right) +a\left( Z\right) S_{ab}\right\} \right] 
\end{eqnarray*}%
But, 
\begin{eqnarray*}
E_{\eta }\left[ b\left( Z\right) \left\{ q\left( O\right) +a\left( Z\right)
S_{ab}\right\} \right]  &=&E_{\eta }\left[ b\left( Z\right) \left[ E_{\eta
}\left\{ q\left( O\right) |Z\right\} +a\left( Z\right) E_{\eta }\left(
S_{ab}|Z\right) \right] \right]  \\
&=&0
\end{eqnarray*}%
where the last identity follows by definition of $a\left( Z\right) .$ The
last assertion of the Theorem follows by Theorem \ref{theo:2}.
\end{proof}

\bigskip

\begin{proof}

Here we prove that the parameter $\protect\psi \left( \protect\eta %
\right) $ in Example \protect\ref{ex:MNAR} is not in the class studied in 
\protect\cite{NeweyCherno19}.

Let $O=\left( DY,D,Z\right) .$ Notice that given $D=0,$ $O$ depends only on $%
Z.$ Let 
\begin{equation*}
\psi \left( \eta \right) \equiv E_{\eta }\left[ \left( 1-D\right) \frac{%
E_{\eta }\left\{ DY\exp \left( \delta Y\right) |Z\right\} }{E_{\eta }\left\{
D\exp \left( \delta Y\right) |Z\right\} }\right] =E_{\eta }\left\{ \left(
1-D\right) a(Z)\right\} 
\end{equation*}%
where $a\left( Z\right) \equiv E_{\eta }\left\{ DY\exp \left( \delta
Y\right) |Z\right\} /E_{\eta }\left\{ D\exp \left( \delta Y\right)
|Z\right\} $ for $\delta \not=0.$ Suppose that there exists $m_{1}^{\ast
}\left( O,\cdot \right) $ such that for each $\eta ,$ the map $h\in
L_{2}\left( P_{\eta ,\left( D,Z\right) }\right) \rightarrow E_{\eta }\left\{
m_{1}^{\ast }\left( O,h\right) \right\} $ is continuous and linear and such
that 
\begin{equation*}
\sigma \left( \eta \right) =E_{\eta }\left\{ m_{1}^{\ast }\left( O,a^{\ast
}\right) \right\} +E_{\eta }\left( S_{0}^{\ast }\right) 
\end{equation*}%
where $a^{\ast }\left( D,Z\right) =E_{\eta }\left\{ \left. q\left( O\right)
\right\vert D,Z\right\} $ for some statistic $q\left( O\right) $.

Without loss of generality we can assume that $q\left( O\right) =Dq^{\ast
}\left( Y,Z\right) $. To see this write $q\left( O\right) =Dq^{\ast }\left(
Y,Z\right) +\left( 1-D\right) q^{\ast \ast }\left( Z\right) .$ Then, 
\begin{eqnarray*}
a^{\ast }\left( D,Z\right)  &=&E_{\eta }\left\{ \left. q\left( O\right)
\right\vert D,Z\right\}  \\
&=&a_{1}^{\ast }\left( D,Z\right) +a_{0}^{\ast }\left( D,Z\right) 
\end{eqnarray*}%
where $a_{1}^{\ast }\left( D,Z\right) \equiv E_{\eta }\left\{ \left.
Dq^{\ast }\left( Y,Z\right) \right\vert D,Z\right\} $ and $a_{0}^{\ast
}\left( D,Z\right) \equiv \left( 1-D\right) q^{\ast \ast }\left( Z\right) $.
Then, by the assumed linearity of the map $h\in L_{2}\left( P_{\eta ,\left(
D,Z\right) }\right) \rightarrow E_{\eta }\left\{ m_{1}^{\ast }\left(
O,h\right) \right\} $ we can now write 
\begin{eqnarray*}
E_{\eta }\left\{ m_{1}^{\ast }\left( O,a^{\ast }\right) \right\} +E_{\eta
}\left( S_{0}^{\ast }\right)  &=&E_{\eta }\left\{ m_{1}^{\ast }\left(
O,a_{1}^{\ast }\right) \right\} +E_{\eta }\left\{ m_{1}^{\ast }\left(
O,a_{0}^{\ast }\right) \right\} +E_{\eta }\left( S_{0}^{\ast }\right)  \\
&=&E_{\eta }\left\{ m_{1}^{\ast }\left( O,a_{1}^{\ast }\right) \right\}
+E_{\eta }\left( S_{0}^{\ast \ast }\right) 
\end{eqnarray*}%
where $S_{0}^{\ast \ast }=m_{1}^{\ast }\left( O,a_{0}^{\ast }\right)
+S_{0}^{\ast }$ is a statistic because $a_{0}^{\ast }\left( D,Z\right) $
does not depend on $\eta $.

So, from now on we will assume $a^{\ast }\left( D,Z\right) \equiv E_{\eta
}\left\{ \left. q\left( O\right) \right\vert D,Z\right\} \,\ $where $q\left(
O\right) =Dq^{\ast }\left( Y,Z\right) $ for some $q^{\ast }$. Note that $%
q\left( Y,Z\right) $ depends on $Y,$ for otherwise $a^{\ast }\left(
D,Z\right) $ would not depend on $\eta .$

Because $\sigma \left( \eta \right) $ is the same functional as $\psi \left(
\eta \right) $ then their unique influence functions $\sigma _{\eta }^{1}$
and $\psi _{\eta }^{1}$ must agree. We shall compute next the influence
function $\psi _{\eta }^{1}$ of $\psi \left( \eta \right) $. For any path $%
t\rightarrow \eta _{t}$ through $\eta _{t=0}=\eta $ with score $g$ we have 
\begin{eqnarray*}
&&\left. \frac{d}{dt}E_{\eta _{t}}\left[ \left( 1-D\right) \frac{E_{\eta
}\left\{ DY\exp \left( \delta Y\right) |Z\right\} }{E_{\eta }\left\{ D\exp
\left( \delta Y\right) |Z\right\} }\right] \right\vert _{t=0} \\
&=&E_{\eta }\left[ \left[ \left( 1-D\right) \frac{E_{\eta }\left\{ DY\exp
\left( \delta Y\right) |Z\right\} }{E_{\eta }\left\{ D\exp \left( \delta
Y\right) |Z\right\} }-\psi \left( \eta \right) \right] g\right] +E_{\eta }%
\left[ E_{\eta }\left( 1-D|Z\right) \frac{\frac{d}{dt}\left. E_{\eta
_{t}}\left\{ DY\exp \left( \delta Y\right) |Z\right\} \right\vert _{t=0}}{%
E_{\eta }\left\{ D\exp \left( \delta Y\right) |Z\right\} }\right]  \\
&&+E_{\eta }\left[ E_{\eta }\left\{ 1-D|Z\right\} E_{\eta }\left\{ DY\exp
\left( \delta Y\right) |Z\right\} \frac{d}{dt}\left. \frac{1}{E_{\eta
_{t}}\left\{ D\exp \left( \delta Y\right) |Z\right\} }\right\vert _{t=0}%
\right]  \\
&=&E_{\eta }\left[ \left\{ \left( 1-D\right) \frac{E_{\eta }\left\{ DY\exp
\left( \delta Y\right) |Z\right\} }{E_{\eta }\left\{ D\exp \left( \delta
Y\right) |Z\right\} }-\psi \left( \eta \right) \right\} g\right]  \\
&&+E_{\eta }\left[ E_{\eta }\left( 1-D|Z\right) \frac{\left[ DY\exp \left(
\delta Y\right) -E_{\eta }\left\{ DY\exp \left( \delta Y\right) |Z\right\} %
\right] }{E_{\eta }\left\{ D\exp \left( \delta Y\right) |Z\right\} }g\right] 
\\
&&+E_{\eta }\left[ E_{\eta }\left\{ 1-D|Z\right\} E_{\eta }\left\{ DY\exp
\left( \delta Y\right) |Z\right\} \frac{\left[ D\exp \left( \delta Y\right)
-E_{\eta }\left\{ D\exp \left( \delta Y\right) |Z\right\} \right] }{E_{\eta
}\left\{ D\exp \left( \delta Y\right) |Z\right\} ^{2}}g\right] .
\end{eqnarray*}%
So, we conclude that 
\begin{eqnarray*}
\psi _{\eta }^{1} &=&\left( 1-D\right) \frac{E_{\eta }\left\{ DY\exp \left(
\delta Y\right) |Z\right\} }{E_{\eta }\left\{ D\exp \left( \delta Y\right)
|Z\right\} }-\psi \left( \eta \right) +E_{\eta }\left( 1-D|Z\right) \frac{%
\left[ DY\exp \left( \delta Y\right) -E_{\eta }\left\{ DY\exp \left( \delta
Y\right) |Z\right\} \right] }{E_{\eta }\left\{ D\exp \left( \delta Y\right)
|Z\right\} } \\
&&-E_{\eta }\left( 1-D|Z\right) E_{\eta }\left\{ DY\exp \left( \delta
Y\right) |Z\right\} \frac{\left[ D\exp \left( \delta Y\right) -E_{\eta
}\left\{ D\exp \left( \delta Y\right) |Z\right\} \right] }{E_{\eta }\left\{
D\exp \left( \delta Y\right) |Z\right\} ^{2}} \\
&=&\left( 1-D\right) \frac{E_{\eta }\left\{ DY\exp \left( \delta Y\right)
|Z\right\} }{E_{\eta }\left\{ D\exp \left( \delta Y\right) |Z\right\} }%
+DY\exp \left( \delta Y\right) \frac{E_{\eta }\left( 1-D|Z\right) }{E_{\eta
}\left\{ D\exp \left( \delta Y\right) |Z\right\} } \\
&&-D\exp \left( \delta Y\right) \frac{E_{\eta }\left( 1-D|Z\right) E_{\eta
}\left\{ DY\exp \left( \delta Y\right) |Z\right\} }{E_{\eta }\left\{ D\exp
\left( \delta Y\right) |Z\right\} ^{2}}-\psi \left( \eta \right) .
\end{eqnarray*}%
On the other hand, letting $\mathcal{R}_{\eta }^{\ast }\left( D,Z\right) $
be the Riesz representer of the map $h\in L_{2}\left( P_{\eta ,\left(
D,Z\right) }\right) \rightarrow E_{\eta }\left\{ m_{1}^{\ast }\left(
O,h\right) \right\} ,$ we have 
\begin{eqnarray*}
&&\left. \frac{d}{dt}E_{\eta _{t}}\left\{ m_{1}^{\ast }\left( O,a_{\eta
_{t}}^{\ast }\right) \right\} +E_{\eta _{t}}\left( S_{0}^{\ast }\right)
\right\vert _{t=0} \\
&=&E_{\eta }\left[ \left[ m_{1}^{\ast }\left( O,a^{\ast }\right) -E_{\eta
}\left\{ m_{1}^{\ast }\left( O,a^{\ast }\right) \right\} \right] g\right]
+\left. \frac{d}{dt}E_{\eta }\left[ \mathcal{R}_{\eta }^{\ast }\left(
D,Z\right) E_{\eta _{t}}\left\{ \left. q\left( O\right) \right\vert
D,Z\right\} \right] \right\vert _{t=0}+E_{\eta }\left[ \left\{ S_{0}^{\ast
}-E_{\eta }\left( S_{0}^{\ast }\right) \right\} g\right]  \\
&=&E_{\eta }\left\{ m_{1}^{\ast }\left( O,a^{\ast }\right) g\right\}
+E_{\eta }\left[ \mathcal{R}_{\eta }^{\ast }\left( D,Z\right) \left[ q\left(
O\right) -E_{\eta }\left\{ \left. q\left( O\right) \right\vert D,Z\right\} %
\right] g\right] +E_{\eta }\left( S_{0}^{\ast }g\right) -\sigma \left( \eta
\right) 
\end{eqnarray*}%
from where we conclude that 
\begin{equation*}
\sigma _{\eta }^{1}=m_{1}^{\ast }\left( O,a^{\ast }\right) +\mathcal{R}%
_{\eta }^{\ast }\left( D,Z\right) \left[ q\left( O\right) -E_{\eta }\left\{
\left. q\left( O\right) \right\vert D,Z\right\} \right] +S_{0}^{\ast
}-\sigma (\eta ).
\end{equation*}%
The uniqueness of influence functions $\sigma _{\eta }^{1}$ and $\psi _{\eta
}^{1}$ implies that 
\begin{eqnarray}
&&\left( 1-D\right) \frac{E_{\eta }\left\{ DY\exp \left( \delta Y\right)
|Z\right\} }{E_{\eta }\left\{ D\exp \left( \delta Y\right) |Z\right\} }%
+DY\exp \left( \delta Y\right) \frac{E_{\eta }\left\{ 1-D|Z\right\} }{%
E_{\eta }\left\{ D\exp \left( \delta Y\right) |Z\right\} }
\label{eq:evaluatea1} \\
&&-D\exp \left( \delta Y\right) \frac{E_{\eta }\left\{ 1-D|Z\right\} E_{\eta
}\left\{ DY\exp \left( \delta Y\right) |Z\right\} }{E_{\eta }\left\{ D\exp
\left( \delta Y\right) |Z\right\} ^{2}}  \label{eq:evaluatea2} \\
&=&m_{1}^{\ast }\left( O,a^{\ast }\right) +\mathcal{R}_{\eta }^{\ast }\left(
D,Z\right) \left\{ q\left( O\right) -E_{\eta }\left[ \left. q\left( O\right)
\right\vert D,Z\right] \right\} +S_{0}^{\ast }  \notag
\end{eqnarray}%
Now, taking $\eta $ and $\eta ^{\prime }$ that agree on the law of $Y,D|Z,$
but disagree on the marginal of $Z,$ the left hand side agree on these two
laws as well as $a^{\ast }$ so, subtracting one from the other we obtain 
\begin{eqnarray*}
0 &=&\left\{ \mathcal{R}_{\eta }^{\ast }\left( D,Z\right) -\mathcal{R}_{\eta
^{\prime }}^{\ast }\left( D,Z\right) \right\} \left[ q\left( O\right)
-E_{\eta }\left\{ \left. q\left( O\right) \right\vert D,Z\right\} \right]  \\
&=&\left\{ \mathcal{R}_{\eta }^{\ast }\left( D=1,Z\right) -\mathcal{R}_{\eta
^{\prime }}^{\ast }\left( D=1,Z\right) \right\} D\left[ q^{\ast }\left(
Y,Z\right) -E_{\eta }\left\{ \left. q^{\ast }\left( Y,Z\right) \right\vert
D=1,Z\right\} \right] 
\end{eqnarray*}%
Since $q^{\ast }\left( Y,Z\right) $ depends on $Y$ then $\mathcal{R}_{\eta
}^{\ast }\left( D=1,Z\right) -\mathcal{R}_{\eta ^{\prime }}^{\ast }\left(
D=1,Z\right) $ must be equal to 0$.$ So, we conclude that $R_{\eta }^{\ast
}\left( D=1,Z\right) $\ depends on $\eta $\ only through the law of $Y,D|Z.$

Next, for any $h\left( D,Z\right) =Du\left( Z\right) ,$ we have 
\begin{equation*}
E_{\eta }\left[ E_{\eta }\left\{ m_{1}^{\ast }\left( O,h\right) |Z\right\} %
\right] =E_{\eta }\left[ E_{\eta }\left\{ \left. \mathcal{R}_{\eta }^{\ast
}\left( D=1,Z\right) Du\left( Z\right) \right\vert Z\right\} \right] \text{
for all }\eta . 
\end{equation*}%
So, since $\mathcal{R}_{\eta }^{\ast }\left( D=1,Z\right) $ does not depend
on the marginal law of $Z,$ we conclude that 
\begin{equation*}
E_{\eta }\left\{ m_{1}^{\ast }\left( O,h\right) |Z\right\} =E_{\eta }\left\{
\left. \mathcal{R}_{\eta }^{\ast }\left( D=1,Z\right) Du\left( Z\right)
\right\vert Z\right\} 
\end{equation*}%
or equivalently 
\begin{align*}
& E_{\eta }\left\{ m_{1}^{\ast }\left( O,h\right) |D=0,Z\right\} E_{\eta
}\left( \left. 1-D\right\vert Z\right) +E_{\eta }\left\{ m_{1}^{\ast }\left(
O,h\right) |D=1,Z\right\} E_{\eta }\left( \left. D\right\vert Z\right) = \\
& \mathcal{R}_{\eta }^{\ast }\left( D=1,Z\right) u\left( Z\right) E_{\eta
}\left( \left. D\right\vert Z\right) .
\end{align*}

Suppose that $z^{\ast }$ is such that $u\left( z^{\ast }\right) =0.$ Then, 
\begin{equation*}
m_{1}^{\ast }\left\{ \left( 0,0,z^{\ast }\right) ,h\right\} +\left[ E_{\eta
}\left\{ m_{1}^{\ast }\left( O,h\right) |D=1,Z=z^{\ast }\right\}
-m_{1}^{\ast }\left\{ \left( 0,0,z^{\ast }\right) ,h\right\} \right] E_{\eta
}\left( \left. D\right\vert Z=z^{\ast }\right) =0
\end{equation*}%
where to arrive at the left hand side we have used the fact that $E_{\eta
}\left\{ m_{1}^{\ast }\left( O,h\right) |D=0,Z\right\} =m_{1}^{\ast }\left\{
\left( 0,0,Z\right) ,h\right\} $. Now, since $E_{\eta }\left\{ m_{1}^{\ast
}\left( O,h\right) |D=1,Z=z^{\ast }\right\} $ does not depend on the law of $%
D|Z,$ then letting $E_{\eta }\left( \left. D\right\vert Z=z^{\ast }\right)
\rightarrow 0$ we conclude that $m_{1}^{\ast }\left\{ \left( 0,0,z^{\ast
}\right) ,h\right\} =0$ and consequently also $E_{\eta }\left\{ m_{1}^{\ast
}\left( O,h\right) |D=1,Z=z^{\ast }\right\} =0.$

Next, for any $Z=z$ such that $u\left( z\right) \not=0$ we write 
\begin{eqnarray*}
&&\frac{1}{E_{\eta }\left( \left. D\right\vert Z=z\right) }\frac{m_{1}^{\ast
}\left\{ \left( 0,0,z\right) ,h\right\} }{u\left( z\right) }+\frac{\left[
E_{\eta }\left\{ m_{1}^{\ast }\left( O,h\right) |D=1,Z=z\right\}
-m_{1}^{\ast }\left\{ \left( 0,0,z\right) ,h\right\} \right] }{u\left(
z\right) } \\
&=&\mathcal{R}_{\eta }^{\ast }\left( D=1,z\right) 
\end{eqnarray*}%
and since $\mathcal{R}_{\eta }^{\ast }\left( D=1,z\right) $ does not depend
on $u$, then taking any other $u^{\ast }$ with $u^{\ast }\left( z\right)
\not=0,$ we have 
\begin{eqnarray*}
0 &=&\frac{1}{E_{\eta }\left( \left. D\right\vert Z=z\right) }\left[ \frac{%
m_{1}^{\ast }\left\{ \left( 0,0,z\right) ,h\right\} }{u\left( z\right) }-%
\frac{m_{1}^{\ast }\left\{ \left( 0,0,z\right) ,h^{\ast }\right\} }{u^{\ast
}\left( z\right) }\right]  \\
&&+\left[ \frac{\left[ E_{\eta }\left\{ m_{1}^{\ast }\left( O,h\right)
|D=1,Z=z\right\} -m_{1}^{\ast }\left\{ \left( 0,0,z\right) ,h\right\} \right]
}{u\left( z\right) }\right]  \\
&&-\left[ \frac{\left[ E_{\eta }\left\{ m_{1}^{\ast }\left( O,h^{\ast
}\right) |D=1,Z=z\right\} -m_{1}^{\ast }\left\{ \left( 0,0,z\right) ,h^{\ast
}\right\} \right] }{u^{\ast }\left( z\right) }\right] .
\end{eqnarray*}%
Once again, since none of the terms in squared brackets depend on the law of 
$D|Z,$ the right hand side is a linear function of $\alpha \equiv 1/E_{\eta
}\left( \left. D\right\vert Z=z\right) $ which can take any value in $\left(
1,\infty \right) ,$ but the left hand side is identically equal to 0.
Therefore, 
\begin{equation*}
\frac{m_{1}^{\ast }\left\{ \left( 0,0,z\right) ,h\right\} }{u\left( z\right) 
}-\frac{m_{1}^{\ast }\left\{ \left( 0,0,z\right) ,h^{\ast }\right\} }{%
u^{\ast }\left( z\right) }=0.
\end{equation*}%
So, we conclude that there exists a function $c\left( z\right) $ independent
of $\eta $ such that for all $h\left( D,Z\right) =Du\left( Z\right) $ 
\begin{equation}
m_{1}^{\ast }\left\{ \left( 0,0,z\right) ,h\right\} =c\left( z\right)
u\left( z\right)   \label{eq:mainu}
\end{equation}

Next, return to the equations \eqref{eq:evaluatea1} and \eqref{eq:evaluatea2}
and evaluate them at $D=0,$ to obtain 
\begin{equation*}
\frac{E_{\eta }\left\{ Y\exp \left( \delta Y\right) |D=1,Z\right\} }{E_{\eta
}\left\{ \exp \left( \delta Y\right) |D=1,Z\right\} }=m_{1}^{\ast }\left(
0,0,Z,a^{\ast }\right) +t\left( Z\right) 
\end{equation*}%
where $t\left( Z\right) $ is equal to $S_{0}^{\ast }\left( 0,0,Z\right) ,$
i.e. to $S_{0}^{\ast }$ evaluated at $D=0.$ Next, recalling that $a^{\ast
}\left( D,Z\right) =DE\left[ q^{\ast }\left( Y,Z\right) |D=1,Z\right] $ and
invoking $\left( \ref{eq:mainu}\right) $ we conclude that 
\begin{equation*}
\frac{E_{\eta }\left\{ Y\exp \left( \delta Y\right) |D=1,Z\right\} }{E_{\eta
}\left\{ \exp \left( \delta Y\right) |D=1,Z\right\} }=c\left( Z\right)
E_{\eta }\left\{ q^{\ast }\left( Y,Z\right) |D=1,Z\right\} +t\left( Z\right) 
\end{equation*}%
where $c\left( z\right) $ and $t\left( z\right) $ are functions of $z$ that
do not depend on $\eta $. We will now show that the last equality cannot
hold for all $\eta $ if $\delta \not=0.$ To do so, we re-write the last
identity as 
\begin{equation}
\frac{E_{\eta }\left\{ DY\exp \left( \delta Y\right) |Z\right\} }{E_{\eta
}\left\{ D\exp \left( \delta Y\right) |Z\right\} }=c\left( Z\right) \frac{%
E_{\eta }\left\{ Dq^{\ast }\left( Y,Z\right) |Z\right\} }{E_{\eta }\left(
D|Z\right) }+t\left( Z\right) .  \label{eq:id1}
\end{equation}%
If this identity holds for all $\eta ,$ then taking expectations on both
sides we have that for all $\eta $ 
\begin{equation}
E_{\eta }\left[ \frac{E_{\eta }\left\{ DY\exp \left( \delta Y\right)
|Z\right\} }{E_{\eta }\left\{ D\exp \left( \delta Y\right) |Z\right\} }%
\right] =E_{\eta }\left[ c\left( Z\right) \frac{E_{\eta }\left\{ Dq^{\ast
}\left( Y,Z\right) |Z\right\} }{E_{\eta }\left( D|Z\right) }\right] +E_{\eta
}\left\{ t\left( Z\right) \right\}   \label{eq:id2}
\end{equation}%
or equivalently, for all $\eta $%
\begin{equation*}
E_{\eta }\left[ \frac{DY\exp \left( \delta Y\right) }{E_{\eta }\left\{ D\exp
\left( \delta Y\right) |Z\right\} }\right] =E_{\eta }\left\{ c\left(
Z\right) \frac{Dq^{\ast }\left( Y,Z\right) }{E_{\eta }\left( D|Z\right) }%
\right\} +E_{\eta }\left\{ t\left( Z\right) \right\} .
\end{equation*}%
Since the functionals on the left and right hand-sides are identical, their
influence functions must agree. Then, taking an arbitrary submodel $%
t\rightarrow \eta _{t}$ with score $g$ at $\eta _{t=0}=\eta $ we have 
\begin{equation*}
\left. \frac{d}{dt}E_{\eta _{t}}\left[ \frac{DY\exp \left( \delta Y\right) }{%
E_{\eta _{t}}\left\{ D\exp \left( \delta Y\right) |Z\right\} }\right]
\right\vert _{t=0}=\left. \frac{d}{dt}E_{\eta _{t}}\left\{ c\left( Z\right) 
\frac{Dq^{\ast }\left( Y,Z\right) }{E_{\eta _{t}}\left( D|Z\right) }\right\}
\right\vert _{t=0}+\left. \frac{d}{dt}E_{\eta _{t}}\left\{ t\left( Z\right)
\right\} \right\vert _{t=0}
\end{equation*}%
from where we conclude that 
\begin{eqnarray*}
&&E_{\eta }\left[ \left[ \frac{DY\exp \left( \delta Y\right) }{E_{\eta
}\left\{ D\exp \left( \delta Y\right) |Z\right\} }-E_{\eta }\left[ \frac{%
DY\exp \left( \delta Y\right) }{E_{\eta }\left\{ D\exp \left( \delta
Y\right) |Z\right\} }\right] \right] g\right]  \\
&&-E_{\eta }\left[ \frac{E_{\eta }\left\{ DY\exp \left( \delta Y\right)
|Z\right\} \left[ D\exp \left( \delta Y\right) -E_{\eta }\left\{ D\exp
\left( \delta Y\right) |Z\right\} \right] }{E_{\eta }\left[ D\exp \left(
\delta Y\right) |Z\right] ^{2}}g\right]  \\
&=&E_{\eta }\left[ \left[ c\left( Z\right) \frac{Dq^{\ast }\left( Y,Z\right) 
}{E_{\eta }\left( D|Z\right) }-E_{\eta }\left\{ c\left( Z\right) \frac{%
Dq^{\ast }\left( Y,Z\right) }{E_{\eta }\left( D|Z\right) }\right\} \right] g%
\right]  \\
&&-E_{\eta }\left[ c\left( Z\right) \frac{E_{\eta }\left\{ \left. Dq^{\ast
}\left( Y,Z\right) \right\vert Z\right\} }{E_{\eta }\left( D|Z\right) ^{2}}%
\left\{ D-E_{\eta }\left( D|Z\right) \right\} g\right] +E_{\eta }\left[
t\left( Z\right) -E_{\eta }\left\{ t\left( Z\right) \right\} g\right] 
\end{eqnarray*}%
Consequently,

\begin{eqnarray*}
&&\left[ \frac{DY\exp \left( \delta Y\right) }{E_{\eta }\left\{ D\exp \left(
\delta Y\right) |Z\right\} }-E_{\eta }\left[ \frac{DY\exp \left( \delta
Y\right) }{E_{\eta }\left\{ D\exp \left( \delta Y\right) |Z\right\} }\right] %
\right]  \\
&&-\frac{E_{\eta }\left\{ DY\exp \left( \delta Y\right) |Z\right\} \left[
D\exp \left( \delta Y\right) -E_{\eta }\left\{ D\exp \left( \delta Y\right)
|Z\right\} \right] }{E_{\eta }\left\{ D\exp \left( \delta Y\right)
|Z\right\} ^{2}} \\
&=&c\left( Z\right) \frac{Dq^{\ast }\left( Y,Z\right) }{E_{\eta }\left(
D|Z\right) }-E_{\eta }\left\{ c\left( Z\right) \frac{Dq^{\ast }\left(
Y,Z\right) }{E_{\eta }\left( D|Z\right) }\right\} -c\left( Z\right) \frac{%
E_{\eta }\left\{ \left. Dq^{\ast }\left( Y,Z\right) \right\vert Z\right\} }{%
E_{\eta }\left( D|Z\right) ^{2}}\left\{ D-E_{\eta }\left[ D|Z\right]
\right\} + \\
&&t\left( Z\right) -E_{\eta }\left\{ t\left( Z\right) \right\} .
\end{eqnarray*}%
Invoking the equalities $\left( \ref{eq:id1}\right) $ and $\left( \ref%
{eq:id2}\right) ,$ the last identity is the same as 
\begin{eqnarray*}
&&\frac{1}{E_{\eta }\left\{ D\exp \left( \delta Y\right) |Z\right\} }DY\exp
\left( \delta Y\right) -\frac{E_{\eta }\left\{ DY\exp \left( \delta Y\right)
|Z\right\} }{E_{\eta }\left\{ D\exp \left( \delta Y\right) |Z\right\} ^{2}}%
D\exp \left( \delta Y\right)  \\
&=&c\left( Z\right) \frac{1}{E_{\eta }\left( D|Z\right) }Dq^{\ast }\left(
Y,Z\right) -Dc\left( Z\right) \frac{E_{\eta }\left\{ \left. Dq^{\ast }\left(
Y,Z\right) \right\vert Z\right\} }{E_{\eta }\left( D|Z\right) ^{2}}
\end{eqnarray*}%
or equivalently%
\begin{eqnarray*}
&&\frac{1}{E_{\eta }\left\{ \exp \left( \delta Y\right) |D=1,Z\right\}
E_{\eta }\left( D|Z\right) }DY\exp \left( \delta Y\right) -\frac{E_{\eta
}\left\{ DY\exp \left( \delta Y\right) |Z\right\} }{E_{\eta }\left\{ \exp
\left( \delta Y\right) |D=1,Z\right\} ^{2}E_{\eta }\left( D|Z\right) ^{2}}%
D\exp \left( \delta Y\right)  \\
&=&c\left( Z\right) \frac{1}{E_{\eta }\left( D|Z\right) }Dq^{\ast }\left(
Y,Z\right) -Dc\left( Z\right) \frac{E_{\eta }\left\{ \left. Dq^{\ast }\left(
Y,Z\right) \right\vert Z\right\} }{E_{\eta }\left( D|Z\right) ^{2}}.
\end{eqnarray*}%
The last equality is equivalent to 
\begin{equation*}
\frac{D\exp \left( \delta Y\right) }{E_{\eta }\left\{ \exp \left( \delta
Y\right) |D=1,Z\right\} }\left[ Y-\frac{E_{\eta }\left\{ Y\exp \left( \delta
Y\right) |D=1,Z\right\} }{E_{\eta }\left\{ \exp \left( \delta Y\right)
|D=1,Z\right\} }\right] =Dc\left( Z\right) \left[ q^{\ast }\left( Y,Z\right)
-E_{\eta }\left\{ \left. q^{\ast }\left( Y,Z\right) \right\vert
D=1,Z\right\} \right] 
\end{equation*}%
The last equation cannot hold for all $\eta .$ To see this, evaluate the
left and right and sides at $y$ and $y^{\ast }$ with $y\not=y^{\ast },$ and
subtract one from the other, to obtain%
\begin{eqnarray}
&&\frac{D\exp \left( \delta y\right) }{E_{\eta }\left\{ \exp \left( \delta
Y\right) |D=1,Z\right\} }\left[ y-\frac{E_{\eta }\left\{ Y\exp \left( \delta
Y\right) |D=1,Z\right\} }{E_{\eta }\left\{ \exp \left( \delta Y\right)
|D=1,Z\right\} }\right]   \notag \\
&&-\frac{D\exp \left( \delta y^{\ast }\right) }{E_{\eta }\left\{ \exp \left(
\delta Y\right) |D=1,Z\right\} }\left[ y^{\ast }-\frac{E_{\eta }\left\{
Y\exp \left( \delta Y\right) |D=1,Z\right\} }{E_{\eta }\left\{ \exp \left(
\delta Y\right) |D=1,Z\right\} }\right]   \label{eq:contra} \\
&=&Dc\left( Z\right) \left\{ q^{\ast }\left( y,Z\right) -q^{\ast }\left(
y^{\ast },Z\right) \right\}   \notag
\end{eqnarray}%
The left hand side depends on $\eta $ whereas the right hand side does not.
We will show that this cannot occur when $\delta \not=0$. To do so, let $%
\eta $ and $\eta ^{\prime }$ correspond to two arbitrary distinct laws. Then
evaluating the left hand side at $\eta $ and at $\eta ^{\prime }$ and
subtracting one from the other, we obtain%
\begin{eqnarray*}
&&D\left[ \frac{1}{E_{\eta }\left\{ \exp \left( \delta Y\right)
|D=1,Z\right\} }-\frac{1}{E_{\eta ^{\prime }}\left\{ \exp \left( \delta
Y\right) |D=1,Z\right\} }\right] \exp \left( \delta y\right) y- \\
&&D\exp \left( \delta y\right) \left[ \frac{E_{\eta }\left\{ Y\exp \left(
\delta Y\right) |D=1,Z\right\} }{E_{\eta }\left\{ \exp \left( \delta
Y\right) |D=1,Z\right\} ^{2}}-\frac{E_{\eta ^{\prime }}\left\{ Y\exp \left(
\delta Y\right) |D=1,Z\right\} }{E_{\eta ^{\prime }}\left\{ \exp \left(
\delta Y\right) |D=1,Z\right\} ^{2}}\right] - \\
&&D\left[ \frac{1}{E_{\eta }\left\{ \exp \left( \delta Y\right)
|D=1,Z\right\} }-\frac{1}{E_{\eta ^{\prime }}\left\{ \exp \left( \delta
Y\right) |D=1,Z\right\} }\right] \exp \left( \delta y^{\ast }\right) y^{\ast
}- \\
&&D\exp \left( \delta y^{\ast }\right) \left[ \frac{E_{\eta }\left\{ Y\exp
\left( \delta Y\right) |D=1,Z\right\} }{E_{\eta }\left\{ \exp \left( \delta
Y\right) |D=1,Z\right\} ^{2}}-\frac{E_{\eta ^{\prime }}\left\{ Y\exp \left(
\delta Y\right) |D=1,Z\right\} }{E_{\eta ^{\prime }}\left\{ \exp \left(
\delta Y\right) |D=1,Z\right\} ^{2}}\right] =0
\end{eqnarray*}%
This holds for all $y$ and $y^{\ast }$. Then, regarding $y^{\ast },\eta $
and $\eta ^{\prime }$ as fixed and $y$ as a free variable the preceding
display is of the form 
\begin{equation*}
k_{1}\left( D,z\right) \exp \left( \delta y\right) y-k_{2}\left( D,z\right)
\exp \left( \delta y\right) +k_{3}\left( D,z\right) =0.
\end{equation*}%
Next, since $\exp \left( \delta y\right) y$ and $\exp \left( \delta y\right) 
$ are not the same function of $y,$ then the preceding identity can only
hold if $k_{j}\left( D,z\right) =0$ for $j=1,2,3.$\ In particular, the
equality $k_{j}\left( D,z\right) =0$ implies that 
\begin{equation*}
E_{\eta }\left\{ \exp \left( \delta Y\right) |D=1,Z\right\} =E_{\eta
^{\prime }}\left\{ \exp \left( \delta Y\right) |D=1,Z\right\} .
\end{equation*}%
But since $\eta $ and $\eta ^{\prime }$ are arbitrary, this implies that $%
E_{\eta }\left\{ \exp \left( \delta Y\right) |D=1,Z\right\} $ does not
depend on $\eta .$ This is a contradiction when $\delta \not=0$ because it
would imply that $\exp \left( \delta Y\right) =c^{\prime }\left( Z\right) $
for some function $c^{\prime }.$ This shows that $\sigma \left( \eta \right) 
$ cannot be equal to $\psi \left( \eta \right) .$

Next, we will show that $\psi \left( \eta \right) $ cannot be equal to any
parameter of the form 
\begin{equation*}
\kappa \left( \eta \right) \equiv E_{\eta }\left\{ m_{1}^{\ast }\left(
O,a^{\ast }\right) \right\} +E_{\eta }\left( S_{0}^{\ast }\right) 
\end{equation*}%
where $a^{\ast }\left( Z\right) =E_{\eta }\left\{ \left. q\left( O\right)
\right\vert Z\right\} $ for some statistic $q\left( O\right) ,$ the map $%
h\in L_{2}\left( P_{\eta ,Z}\right) \rightarrow E_{\eta }\left\{ m_{1}^{\ast
}\left( O,h\right) \right\} $ is continuous and linear for each $\eta $ and $%
S_{0}^{\ast }$ is a statistic. To proceed, just as before, we start by
arguing that if the parameter $\kappa \left( \eta \right) $ is the same as $%
\psi \left( \eta \right) $ then their unique influence functions must agree.
We have already computed the influence function $\psi _{\eta }^{1}$ of $\psi
\left( \eta \right) .$

On the other hand, it is easy to see that the influence function $\kappa
_{\eta }^{1}$ of $\kappa \left( \eta \right) $ is 
\begin{equation*}
\kappa _{\eta }^{1}=m_{1}^{\ast }\left( O,a^{\ast }\right) +\mathcal{R}%
_{\eta }^{\ast }\left( Z\right) \left\{ q\left( O\right) -E_{\eta }\left\{
\left. q\left( O\right) \right\vert Z\right\} \right\} +S_{0}^{\ast }-\kappa
\left( \eta \right) 
\end{equation*}%
where $\mathcal{R}_{\eta }^{\ast }\left( Z\right) $ is the Riesz representer
of the map $h\in L_{2}\left( P_{\eta ,Z}\right) \rightarrow E_{\eta }\left\{
m_{1}^{\ast }\left( O,h\right) \right\} .$ Consequently, equating $\kappa
_{\eta }^{1}$ with $\psi _{\eta }^{1}$ we obtain 
\begin{eqnarray}
&&\left( 1-D\right) \frac{E_{\eta }\left\{ DY\exp \left( \delta Y\right)
|Z\right\} }{E_{\eta }\left\{ D\exp \left( \delta Y\right) |Z\right\} }%
+DY\exp \left( \delta Y\right) \frac{E_{\eta }\left( 1-D|Z\right) }{E_{\eta
}\left\{ D\exp \left( \delta Y\right) |Z\right\} }-  \label{eq:evaluateb1} \\
&&D\exp \left( \delta Y\right) \frac{E_{\eta }\left( 1-D|Z\right) E_{\eta
}\left\{ DY\exp \left( \delta Y\right) |Z\right\} }{E_{\eta }\left\{ D\exp
\left( \delta Y\right) |Z\right\} ^{2}}  \label{eq:evaluateb2} \\
&=&m_{1}^{\ast }\left( O,a^{\ast }\right) +\mathcal{R}_{\eta }^{\ast }\left(
Z\right) \left[ q\left( O\right) -E_{\eta }\left\{ \left. q\left( O\right)
\right\vert Z\right\} \right] +S_{0}^{\ast }.  \notag
\end{eqnarray}%
Now, taking $\eta $ and $\eta ^{\prime }$ that agree on the law of $Y,D|Z,$
but disagree on the marginal of $Z,$ the left hand side agree on these two
laws as well as $a^{\ast }$ so, subtracting one from the other we obtain 
\begin{equation*}
0=\left\{ \mathcal{R}_{\eta }^{\ast }\left( Z\right) -\mathcal{R}_{\eta
^{\prime }}^{\ast }\left( Z\right) \right\} \left[ q\left( O\right) -E_{\eta
}\left\{ \left. q\left( O\right) \right\vert Z\right\} \right] .
\end{equation*}%
Since $q\left( O\right) $ depends on $\left( D,Y\right) $ then $\left\{ 
\mathcal{R}_{\eta }^{\ast }\left( Z\right) -\mathcal{R}_{\eta ^{\prime
}}^{\ast }\left( Z\right) \right\} $ must be equal to 0$.$ So, we conclude
that $R_{\eta }^{\ast }\left( Z\right) $\ depends on $\eta $\ only through
the law of $Y,D|Z.$

Next, for any $h\left( Z\right) ,$ we have 
\begin{equation*}
E_{\eta }\left[ E_{\eta }\left\{ m_{1}^{\ast }\left( O,h\right) |Z\right\} %
\right] =E_{\eta }\left\{ \mathcal{R}_{\eta }^{\ast }\left( Z\right) h\left(
Z\right) \right\} \text{ for all }\eta .
\end{equation*}%
So, since $\mathcal{R}_{\eta }^{\ast }\left( Z\right) $ does not depend on
the marginal law of $Z,$ we conclude that 
\begin{equation*}
E_{\eta }\left\{ m_{1}^{\ast }\left( O,h\right) |Z\right\} =\mathcal{R}%
_{\eta }^{\ast }\left( Z\right) h\left( Z\right) .
\end{equation*}%
The last equality is the same as 
\begin{equation*}
m_{1}^{\ast }\left( 0,0,Z,h\right) E_{\eta }\left( 1-D|Z\right) +E_{\eta
}\left\{ m_{1}^{\ast }\left( O,h\right) |D=1,Z\right\} E_{\eta }\left(
D|Z\right) =\mathcal{R}_{\eta }^{\ast }\left( Z\right) h\left( Z\right) 
\end{equation*}%
or equivalently 
\begin{equation*}
m_{1}^{\ast }\left( 0,0,Z,h\right) +\left[ E_{\eta }\left\{ m_{1}^{\ast
}\left( O,h\right) |D=1,Z\right\} -m_{1}^{\ast }\left( 0,0,Z,h\right) \right]
E_{\eta }\left( D|Z\right) =\mathcal{R}_{\eta }^{\ast }\left( Z\right)
h\left( Z\right) 
\end{equation*}

If $z^{\ast }$ is such that $h\left( z^{\ast }\right) =0,$ then 
\begin{equation*}
m_{1}^{\ast }\left( 0,0,z^{\ast },h\right) +\left[ E_{\eta }\left\{
m_{1}^{\ast }\left( O,h\right) |D=1,Z=z^{\ast }\right\} -m_{1}^{\ast }\left(
0,0,z^{\ast },h\right) \right] E_{\eta }\left( D|Z=z^{\ast }\right) =0 
\end{equation*}%
and since $E_{\eta }\left\{ m_{1}^{\ast }\left( O,h\right) |D=1,Z=z^{\ast
}\right\} $ does not depend on the law of $D|Z,$ then since $E_{\eta }\left(
D|Z=z^{\ast }\right) $ can take any value in $\left( 0,1\right) ,$ we
conclude that $m_{1}^{\ast }\left( 0,0,z^{\ast },h\right) =0$.

On the other hand, for $z$ such that $h\left( z\right) \not=0,$ we have

\begin{equation*}
\frac{m_{1}^{\ast }\left( 0,0,z,h\right) }{h\left( z\right) }+\frac{\left[
E_{\eta }\left\{ m_{1}^{\ast }\left( O,h\right) |D=1,Z=z\right\}
-m_{1}^{\ast }\left( 0,0,z,h\right) \right] }{h\left( z\right) }E_{\eta
}\left( D|Z=z\right) =\mathcal{R}_{\eta }^{\ast }\left( Z=z\right) 
\end{equation*}%
Consequently, for any other $h^{\ast }$ such that $h\left( z^{\ast }\right)
\not=0,$ we have 
\begin{eqnarray*}
0 &=&\left\{ \frac{m_{1}^{\ast }\left( 0,0,z,h\right) }{h\left( z\right) }-%
\frac{m_{1}^{\ast }\left( 0,0,z,h^{\ast }\right) }{h^{\ast }\left( z\right) }%
\right\} + \\
&&\frac{\left[ E_{\eta }\left\{ m_{1}^{\ast }\left( O,h\right)
|D=1,Z=z\right\} -m_{1}^{\ast }\left( 0,0,z,h\right) \right] }{h\left(
z\right) }E_{\eta }\left( D|Z=z\right) - \\
&&\frac{\left[ E_{\eta }\left\{ m_{1}^{\ast }\left( O,h^{\ast }\right)
|D=1,Z=z\right\} -m_{1}^{\ast }\left( 0,0,z,h^{\ast }\right) \right] }{%
h^{\ast }\left( z\right) }E_{\eta }\left( D|Z=z\right) .
\end{eqnarray*}%
Once again, since $E_{\eta }\left\{ m_{1}^{\ast }\left( O,h^{\ast }\right)
|D=1,Z=z\right\} $ and $E_{\eta }\left\{ m_{1}^{\ast }\left( O,h\right)
|D=1,Z=z\right\} $ do not depend on the law of $D|Z,$ and since $E_{\eta
}\left( D|Z=z\right) $ can take any value in $\left( 0,1\right) $ we
conclude that 
\begin{equation*}
\frac{m_{1}^{\ast }\left( 0,0,Z,h\right) }{h\left( Z\right) }-\frac{%
m_{1}^{\ast }\left( 0,0,Z,h^{\ast }\right) }{h^{\ast }\left( Z\right) }=0
\end{equation*}%
Consequently, there exists a function $c\left( Z\right) $ such that for all $%
h$%
\begin{equation}
m_{1}^{\ast }\left( 0,0,Z,h\right) =c\left( Z\right) h\left( Z\right) .
\label{eq:otra}
\end{equation}%
Now, evaluating \eqref{eq:evaluateb1} and \eqref{eq:evaluateb2} at $D=0$ 
\begin{equation*}
\frac{E_{\eta }\left\{ DY\exp \left( \delta Y\right) |Z\right\} }{E_{\eta
}\left\{ D\exp \left( \delta Y\right) |Z\right\} }=m_{1}^{\ast }\left(
0,0,Z,a^{\ast }\right) +\mathcal{R}_{\eta }^{\ast }\left( Z\right) \underset{%
=0}{\underbrace{\left[ q\left( 0,0,Z\right) -E_{\eta }\left\{ \left. q\left(
0,0,Z\right) \right\vert Z\right\} \right] }}+S_{0}^{\ast }\left(
0,0,Z\right) .
\end{equation*}%
So, with $t\left( Z\right) \equiv S_{0}^{\ast }\left( 0,0,Z\right) $ and
with $a^{\ast }\left( Z\right) =E_{\eta }\left\{ \left. q\left( O\right)
\right\vert Z\right\} $ substituted for $h$ in $\left( \ref{eq:otra}\right) $
we arrive at the conclusion that the following equality must hold for all $%
\eta $ 
\begin{equation*}
\frac{E_{\eta }\left\{ DY\exp \left( \delta Y\right) |Z\right\} }{E_{\eta
}\left\{ D\exp \left( \delta Y\right) |Z\right\} }=c\left( Z\right) E_{\eta
}\left\{ \left. q\left( O\right) \right\vert Z\right\} +t\left( Z\right) .
\end{equation*}%
Therefore, 
\begin{equation*}
E_{\eta }\left[ \frac{DY\exp \left( \delta Y\right) }{E_{\eta }\left\{ D\exp
\left( \delta Y\right) |Z\right\} }\right] =E_{\eta }\left\{ c\left(
Z\right) q\left( O\right) +t\left( Z\right) \right\} .
\end{equation*}%
Again equating the influence functions of the functionals on the right and
left hand sides we conclude that 
\begin{equation*}
\frac{DY\exp \left( \delta Y\right) }{E_{\eta }\left\{ D\exp \left( \delta
Y\right) |Z\right\} }-\frac{E_{\eta }\left\{ YD\exp \left( \delta Y\right)
|Z\right\} }{E_{\eta }\left\{ D\exp \left( \delta Y\right) |Z\right\} ^{2}}%
\left[ D\exp \left( \delta Y\right) -E_{\eta }\left\{ D\exp \left( \delta
Y\right) |Z\right\} \right] =c\left( Z\right) q\left( O\right) +t\left(
Z\right) 
\end{equation*}%
which leads to a contradiction when $\delta \not=0.$ To arrive at the
contradiction we would reason just as we did to show that the equality $%
\left( \ref{eq:contra}\right) $ leads to a contradiction.
\end{proof}

\begin{proof}

Here we prove that the parameter in Example \protect\ref{ex:APE},
is in the class of \protect\cite{NeweyCherno19} but not in the class of 
\protect\cite{Robins08Higher}.

Let  
\begin{equation*}
\chi \left( \eta \right) \equiv E_{\eta }\left\{ \int_{0}^{1}E_{\eta }\left(
Y|D=u,L\right) w\left( u\right) du\right\} .
\end{equation*}%
Its influence function is 
\begin{equation*}
\chi _{\eta }^{1}=\int_{0}^{1}E_{\eta }\left( Y|D=u,L\right) w\left(
u\right) du+\frac{w\left( D\right) }{f_{\eta }\left( D|L\right) }\left\{
Y-E_{\eta }\left( Y|D,L\right) \right\} -\chi \left( \eta \right) .
\end{equation*}%
Suppose the parameter $\chi \left( \eta \right) $ is in the class of \cite%
{Robins08Higher} for some functions $a^{\ast }\left( D,L\right) $ and $%
b^{\ast }\left( D,L\right) $ which are non-constant in $D$ and in $L.$ Then,
there would exist statistics $S_{ab},S_{a},S_{b}$ and $S_{0}$ such that 
\begin{equation*}
a^{\ast }\left( D,L\right) =-\frac{E_{\eta }\left( S_{b}|D,L\right) }{%
E_{\eta }\left( S_{ab}|D,L\right) },b^{\ast }\left( D,L\right) =--\frac{%
E_{\eta }\left( S_{a}|D,L\right) }{E_{\eta }\left( S_{ab}|D,L\right) }
\end{equation*}%
and such that%
\begin{equation*}
\chi _{\eta }^{1}=S_{ab}a^{\ast }\left( D,L\right) b^{\ast }\left(
D,L\right) +S_{a}a^{\ast }\left( D,L\right) +S_{b}b^{\ast }\left( D,L\right)
+S_{0}-\chi \left( \eta \right) .
\end{equation*}

Then, equating the influence functions we arrive at 
\begin{align*}
&\int_{0}^{1}E_{\eta }\left( Y|D=u,L\right) w\left( u\right) du+\frac{%
w\left( D\right) }{f_{\eta }\left( D|L\right) }\left\{ Y-E_{\eta }\left(
Y|D,L\right) \right\} \\
&=S_{ab}a^{\ast }\left( D,L\right) b^{\ast }\left( D,L\right) +S_{a}a^{\ast
}\left( D,L\right) +S_{b}b^{\ast }\left( D,L\right) +S_{0}
\end{align*}%
The right hand side does not depend on $f_{\eta }\left( D|L\right) .$ On the
other hand, in the left hand side 
\begin{equation*}
\int_{0}^{1}E_{\eta }\left( Y|D=u,L\right) w\left( u\right) du 
\end{equation*}
does not depend on $f_{\eta }\left( D|L\right) $ but $\frac{w\left( D\right) 
}{f_{\eta }\left( D|L\right) }\left\{ Y-E_{\eta }\left( Y|D,L\right)
\right\} $ depends on $f_{\eta }\left( D|L\right) $. This is a contradiction.

Now, suppose that we re-define $Z=L,$ so that we now partition $O$ into $%
\left( Y,D\right) $ and $Z=L.$ If the parameter $\chi \left( \eta \right) $
was in the class of \cite{Robins08Higher} for some functions $a^{\ast
}\left( L\right) $ and $b^{\ast }\left( L\right) $ which are non-constant in 
$L,$ then, there would exist statistics $S_{ab},S_{a},S_{b}$ and $S_{0}$
such that 
\begin{equation*}
a^{\ast }\left( L\right) =-\frac{E_{\eta }\left( S_{b}|L\right) }{E_{\eta
}\left( S_{ab}|L\right) },b^{\ast }\left( L\right) =--\frac{E_{\eta }\left(
S_{a}|L\right) }{E_{\eta }\left( S_{ab}|L\right) } 
\end{equation*}%
and such that%
\begin{equation*}
\chi _{\eta }^{1}=S_{ab}a^{\ast }\left( L\right) b^{\ast }\left( L\right)
+S_{a}a^{\ast }\left( L\right) +S_{b}b^{\ast }\left( L\right) +S_{0}-\chi
\left( \eta \right). 
\end{equation*}

Then equating the influence functions we would arrive at 
\begin{equation*}
\int_{0}^{1}E_{\eta }\left( Y|D=u,L\right) w\left( u\right) du+\frac{w\left(
D\right) }{f_{\eta }\left( D|L\right) }\left\{ Y-E_{\eta }\left(
Y|D,L\right) \right\} =S_{ab}a^{\ast }\left( L\right) b^{\ast }\left(
L\right) +S_{a}a^{\ast }\left( L\right) +S_{b}b^{\ast }\left( L\right)
+S_{0}.
\end{equation*}%
This is also a contradiction because, for each fixed $D=d,L=l$ the right
hand side depends on the value of $f_{\eta }\left( d|l\right) ,$ i.e. on $%
f_{\eta }\left( D|L\right) $ evaluated at $D=d,L=l,$ but the right hand side
does not, since for a given law of $Y|D,L,$ one can always modify $f_{\eta
}\left( D|L\right) $ at a point $D=d,L=l$ and still obtain the same values
of $E_{\eta }\left( S_{b}|L=l\right) ,E_{\eta }\left( S_{a}|L=l\right) $ and 
$E_{\eta }\left( S_{ab}|L=l\right) .$
\end{proof}

\bibliographystyle{apalike}
\bibliography{non-linear}

\end{document}